\documentclass{amsart}

\title[Class choice and the surprising weakness of \KM]{Class choice and the surprising weakness of Kelley-Morse set theory}

\author{Victoria Gitman}
\address[V. Gitman]{The City University of New York, CUNY Graduate Center, Mathematics Program, 365 Fifth Avenue, New York, NY 10016}
\email{vgitman@gmail.com}
\urladdr{https://victoriagitman.github.io/}

\author{Joel David Hamkins}
\address[Joel David Hamkins]
{O'Hara Professor of Logic, University of Notre Dame, 100 Malloy Hall, Notre Dame, IN 46556 USA}
\email{jdhamkins@nd.edu}
\urladdr{http://jdh.hamkins.org}

\author{Thomas A. Johnstone}
\address{T. A. Johnstone, Mathematics, New York City College of Technology, 300 Jay Street, Brooklyn, NY 11201}
\email{tjohnstone@citytech.cuny.edu}
\thanks{}
\usepackage{latexsym,amsfonts,amsmath,amssymb,hyperref,mathrsfs}
\usepackage{verbatim}
\usepackage{bbm}

\newtheorem{theorem}{Theorem}
\newtheorem*{theorem*}{Theorem}

\newtheorem*{maintheorem*}{Main Theorem}
\newtheorem*{maintheorems*}{Main Theorems}

\newtheorem*{corollary*}{Corollary}
\newtheorem*{corollaries*}{Corollaries}

\newtheorem{lemma}[theorem]{Lemma}

\newtheorem{observation}[theorem]{Observation}

\theoremstyle{definition}
\newtheorem{definition}[theorem]{Definition}
\newtheorem*{definition*}{Definition}

\newtheorem{question}[theorem]{Question}
\newtheorem*{question*}{Question}

\newtheorem*{questions*}{Questions}
\newtheorem*{mainquestion*}{Main Question} % without numbering
\newtheorem*{openquestion*}{Open Question} % without numbering
\theoremstyle{remark}

\theoremstyle{theorem}% in case new theorems are declared in main.
\newcommand{\QED}{\end{proof}}

\def\proclaim[#1]{{\bf #1}}
\def\BF#1.{{\bf #1.}}

\def\says#1:#2\par{\item[#1] #2\par}
\newcommand{\Los}{\L o\'s}

\newcommand{\Godel}{G\"odel}

\newcommand{\Levy}{L\'{e}vy}
\newcommand{\Lowenheim}{L\"owenheim}

\renewcommand{\P}{{\mathbb P}}
\newcommand{\Q}{{\mathbb Q}}

\newcommand{\one}{\mathbbm{1}} % requires \usepackage{bbm}
\makeatletter
\newcommand{\dotminus}{\mathbin{\text{\@dotminus}}}
\newcommand{\@dotminus}{%
  \ooalign{\hidewidth\raise1ex\hbox{.}\hidewidth\cr$\m@th-$\cr}%
}
\makeatother

\newcommand{\of}{\subseteq}

\newcommand{\set}[1]{\{\,{#1}\,\}}

\newcommand{\Add}{\mathop{\rm Add}}

\newcommand{\Coll}{\mathop{\rm Coll}}
\newcommand{\Ult}{\mathop{\rm Ult}}

\newcommand{\Con}{\mathop{{\rm Con}}}

\newcommand{\image}{\mathbin{\hbox{\tt\char'42}}}

\newcommand{\restrict}{\upharpoonright} % uses amssymb
\newcommand{\satisfies}{\models}
\newcommand{\forces}{\Vdash}

\newcommand\dbrace{\hskip-1.5em\raise3pt\hbox{\rotatebox[origin=c]{-35}{$\left.\strut^{\phantom{|}}\right\}$}}}% useful for tetration

\newcommand\UParroW{{\setbox0\hbox{$\Uparrow$}\rlap{\hbox to \wd0{\hss$\mid$\hss}}\box0}}
\renewcommand{\setminus}{\raise.3ex\hbox{\rotatebox{-20}{$-$}}} % the usual setminus is absurdly huge and vertical

\renewcommand{\emptyset}{\varnothing}

\newcommand{\Union}{\bigcup}

\newcommand{\smalllt}{\mathrel{\mathchoice{\raise2pt\hbox{$\scriptstyle<$}}{\raise1pt\hbox{$\scriptstyle<$}}{\raise0pt\hbox{$\scriptscriptstyle<$}}{\scriptscriptstyle<}}}
\newcommand{\smallleq}{\mathrel{\mathchoice{\raise2pt\hbox{$\scriptstyle\leq$}}{\raise1pt\hbox{$\scriptstyle\leq$}}{\raise1pt\hbox{$\scriptscriptstyle\leq$}}{\scriptscriptstyle\leq}}}
\newcommand{\lt}{\smalllt}

\newcommand{\ltkappa}{{{\smalllt}\kappa}}

\newcommand{\leqdelta}{{{\smallleq}\delta}}

   \def\DHLhksqrt#1#2{%
   \setbox0=\hbox{$#1\sqrt{#2\,}$}\dimen0=\ht0
   \advance\dimen0-0.2\ht0
   \setbox2=\hbox{\vrule height\ht0 depth -\dimen0}%
   {\box0\lower0.4pt\box2}}

\def\[#1]{\mathopen{\lbrack\!\lbrack}#1\mathclose{\rbrack\!\rbrack}}
\newbox\gnBoxA
\newbox\gnBoxB
\newdimen\gnCornerHgt
\setbox\gnBoxA=\hbox{\tiny$\ulcorner$}
\global\gnCornerHgt=\ht\gnBoxA
\newdimen\gnArgHgt
\def\gcode #1{%
\setbox\gnBoxA=\hbox{$#1$}%
\setbox\gnBoxB=\hbox{$\bar #1$}%
\gnArgHgt=\ht\gnBoxB%
\ifnum     \gnArgHgt<\gnCornerHgt \gnArgHgt=0pt%
\else \advance \gnArgHgt by -\gnCornerHgt%
\fi \raise\gnArgHgt\hbox{\tiny$\ulcorner$} \box\gnBoxA %
\raise\gnArgHgt\hbox{\tiny$\urcorner$}}
\newcommand{\UnderTilde}[1]{{\setbox1=\hbox{$#1$}\baselineskip=0pt\vtop{\hbox{$#1$}\hbox to\wd1{\hfil$\sim$\hfil}}}{}}
\newcommand{\Undertilde}[1]{{\setbox1=\hbox{$#1$}\baselineskip=0pt\vtop{\hbox{$#1$}\hbox to\wd1{\hfil$\scriptstyle\sim$\hfil}}}{}}
\newcommand{\undertilde}[1]{{\setbox1=\hbox{$#1$}\baselineskip=0pt\vtop{\hbox{$#1$}\hbox to\wd1{\hfil$\scriptscriptstyle\sim$\hfil}}}{}}
\newcommand{\UnderdTilde}[1]{{\setbox1=\hbox{$#1$}\baselineskip=0pt\vtop{\hbox{$#1$}\hbox to\wd1{\hfil$\approx$\hfil}}}{}}
\newcommand{\Underdtilde}[1]{{\setbox1=\hbox{$#1$}\baselineskip=0pt\vtop{\hbox{$#1$}\hbox to\wd1{\hfil\scriptsize$\approx$\hfil}}}{}}

\def\<#1>{\left\langle#1\right\rangle}

\newcommand{\Tr}{\mathop{\rm Tr}\nolimits}

\newcommand{\ORD}{\mathord{{\rm Ord}}}
\newcommand{\Ord}{\mathord{{\rm Ord}}}
\newcommand{\ZFC}{{\rm ZFC}}
\newcommand{\ZF}{{\rm ZF}}

\newcommand{\KM}{{\rm KM}}

\newcommand{\GBC}{{\rm GBC}}

\newcommand{\GCH}{{\rm GCH}}

\newcommand{\AC}{{\rm AC}}

\newcommand{\HOD}{{\rm HOD}}

\newcommand{\PA}{{\rm PA}}

\newcommand{\cell}[1]{\boxit{\hbox to 17pt{\strut\hfil$#1$\hfil}}}
\newcommand{\head}[2]{\lower2pt\vbox{\hbox{\strut\footnotesize\it\hskip3pt#2}\boxit{\cell#1}}}
\newcommand{\boxit}[1]{\setbox4=\hbox{\kern2pt#1\kern2pt}\hbox{\vrule\vbox{\hrule\kern2pt\box4\kern2pt\hrule}\vrule}}
\newcommand{\Col}[3]{\hbox{\vbox{\baselineskip=0pt\parskip=0pt\cell#1\cell#2\cell#3}}}
\newcommand{\tapenames}{\raise 5pt\vbox to .7in{\hbox to .8in{\it\hfill input: \strut}\vfill\hbox to
.8in{\it\hfill scratch: \strut}\vfill\hbox to .8in{\it\hfill output: \strut}}}
\newcommand{\Head}[4]{\lower2pt\vbox{\hbox to25pt{\strut\footnotesize\it\hfill#4\hfill}\boxit{\Col#1#2#3}}}
\newcommand{\Dots}{\raise 5pt\vbox to .7in{\hbox{\ $\cdots$\strut}\vfill\hbox{\ $\cdots$\strut}\vfill\hbox{\
$\cdots$\strut}}}
\theoremstyle{definition}

\newcommand{\fa}{{\rm Z}_2}
\newcommand{\Cl}{\mathcal S}
\newcommand{\V}{\mathcal V}
\newcommand{\sym}{{\rm sym}}
\newcommand{\Fil}{\mathcal F}
\newcommand{\HS}{{\rm HS}}

\newcommand{\s}{\mathbb S}
\begin{document}

\maketitle

\begin{abstract}
Kelley-Morse set theory \KM\ is weaker than generally supposed and fails to prove several principles that may be desirable in a foundational second-order set theory. Even though \KM\ includes the global choice principle, for example, (i) \KM\ does not prove the \emph{class choice} scheme, asserting that whenever every set $x$ admits a class $X$ with $\varphi(x,X)$, then there is a class $Z\of V\times V$ for which $\varphi(x,Z_x)$ on every section. This scheme can fail with \KM\ even in low-complexity first-order instances $\varphi$ and even when only a set of indices $x$ are relevant. For closely related reasons, (ii) the theory \KM\ does not prove the \Los\ theorem scheme for internal second-order ultrapowers, even for large cardinal ultrapowers, such as the ultrapower by a normal measure on a measurable cardinal. Indeed, the theory \KM\ itself is not generally preserved by internal ultrapowers. Finally, (iii) \KM\ does not prove that the $\Sigma^1_n$ logical complexity is invariant under first-order quantifiers, even  bounded first-order quantifiers. For example, $\forall \alpha{<}\delta\ \psi(\alpha,X)$ is not always provably equivalent to a $\Sigma^1_1$ assertion when $\psi$ is. Nevertheless, these various weaknesses in \KM\ are addressed by augmenting it with the class choice scheme, thereby forming the theory $\KM^+$, which we propose as a robust \KM\ alternative for the foundations of second-order set theory.
\end{abstract}

\section{Introduction}

Suppose that for every natural number $n$ there is a class $X$ for which some property $\varphi(n,X)$ holds---assume we are working in a suitable set theory with classes, including the axiom of choice and even the global choice principle. Should we expect to find a uniformizing class $Z\of\omega\times V$ for which $\varphi(n,Z_n)$ holds for all the various sections? Similarly, if every ordinal $\alpha$ has a class $X$ for which $\varphi(\alpha,X)$ holds, then should we we expect to find a class $Z\of\Ord\times V$ for which $\varphi(\alpha,Z_\alpha)$ holds on every section? Or if every set $x$ admits a class $X$ with $\varphi(x,X)$, then will we find a class $Z\of V\times V$ which unifies these witnesses $\varphi(x,Z_x)$? 

These would all be instances of the \emph{class choice} scheme, a central concern of this article. The class $Z$ in each case uniformizes the class witnesses, in effect choosing for each index $x$ a class witness $Z_x$ on that section. The class choice scheme is therefore a choice principle enabling a choice of classes, rather than sets. 

It is relatively easy to see that the class choice scheme is not provable in \Godel-Bernays \GBC\ set theory, even with the axiom of global choice being available in that theory. A countermodel is provided by any transitive model of $\ZFC+V=\HOD$, equipped with all and only its definable classes. In such a model, we may observe that for each $n$ there is a definable $\Sigma_n$-truth predicate (a partial satisfaction class), but by Tarski's theorem on the non-definability of truth, these cannot be unified into a single definable class. Indeed, in this example we are not really ``choosing'' the classes at all, since for each $n$ there is in fact a unique $\Sigma_n$-truth predicate. Rather, what is going on here is that \GBC\ is not strong enough (although \KM\ is) to collect these various unique witnesses together into a single uniform class. The situation resembles more a class-wise failure of collection or even replacement than choice.

Perhaps it will be a little more surprising, however, to learn that the class choice scheme is not provable even in the considerably stronger Kelley-Morse \KM\ set theory, a theory for which the simple truth-predicate counterexample above is no problem at all, as \KM\ proves the existence of uniform first-order truth predicate satisfaction classes relative to any class parameter. Nevertheless, in section \ref{sec:independence} we show that \KM\ does not prove the class choice scheme, and indeed the scheme can fail even in very low-complexity first-order instances and again with merely a set of indices. So although Kelley-Morse set theory \KM\ proves the existence of first-order truth predicates, it does not prove the class choice scheme, even in easy-seeming cases.

Furthermore, the proof methods reveal several further surprising weaknesses of \KM, which may call into question the suitability of this theory as a foundation of set theory. Namely, \KM\ does not prove the \Los\ theorem scheme for internal second-order ultrapowers, even in the case of large-cardinal ultrapowers, such as those arising from a normal measure on a measurable cardinal---yes, the familiar large cardinal embeddings can fail to be elementary in the language of \KM\ set theory. In the case of ultrapowers by an ultrafilter on $\omega$, we show, the internal ultrapower of the \KM\ universe is not itself always even a model of \KM. Finally, we will show that \KM\ does not prove that second-order logical complexity is preserved by first-order quantifiers, even bounded first-order quantifiers. In light of all these weaknesses, the theory \KM\ appears not quite as robust a foundational theory as might have been thought.

Meanwhile, by augmenting the theory \KM\ with the class choice scheme, thereby forming the theory we denote $\KM^+$, we addresses all these weaknesses, and in this sense $\KM^+$ succeeds as a robust foundational treatment of second-order set theory.

\subsection{Background on sets and classes}

The objects of first-order set theory are sets, and set theorists commonly take the classes of a model of first-order set theory to be simply the definable collections of sets (allowing parameters), which means that the study of their properties naturally takes place in the meta-theory. A formal framework for studying sets and classes, in contrast, is provided by second-order set theory. The most natural interpretation of second-order set theory is in a two-sorted logic with separate sorts (separate variables and quantifiers) for sets and classes. The language of second-order set theory has two membership relations, one deciding set membership in sets and the other deciding set membership in classes. Because foundational axioms of second-order set theory assert that classes are extensional collections of sets, the models of second-order set theory relevant for us have the form $\< V,\in,\Cl>$, where $\<V,\in>$ is a model of the language of first-order set theory and $\Cl$ is the family of classes, that is, collections of sets.\footnote{Although two-sorted logic provides a heuristically preferred interpretation for second-order set theory, it can in fact be formalized in one-sorted first-order logic. In the first-order interpretation, the objects are classes with a membership relation, and sets are defined to be those classes which are elements of other classes.} Following the standard convention, we use lowercase letters for the sets and uppercase letters for the classes. In the meta-theory of second-order set theory, we can also study \emph{hyperclasses}, the definable (with parameters) collections of classes.

Any reasonable axiomatic foundation for second-order set theory should assert that the collection of all sets satisfies $\ZFC$ and that, more generally, it continues to satisfy $\ZFC$ in the extended first-order language with predicates for any finitely many classes. The more general requirement translates into the class replacement axiom, which states that the image of a set under a class function is itself a set. A reasonable foundation should also include some class existence principles and should in particular generalize the notion of classes from first-order set theory by asserting that every first-order definable collection is a class. The two most commonly considered axiomatic foundations for second-order set theory are the \Godel-Bernays axioms ($\GBC$) and the Kelley-Morse axioms ($\KM$). They are distinguished by the strength of their comprehension axiom schemes that decide which (second-order) definable collections are classes. The theory $\GBC$ includes only the minimal first-order comprehension and it is equiconsistent with $\ZFC$. The theory $\KM$ asserts full second-order comprehension (see the next section~\ref{sec:KM} for precise axiomatizations) and strength-wise it lies between the existence of a transitive model of $\ZFC$ and the existence of an inaccessible cardinal (see lemmas~\ref{le:transitiveModelKM} and~\ref{le:inaccessibleKM}).

\subsection{Class choice schemes}

Axiomatizations of second-order set theory often include choice principles for classes because these imply many properties desirable in this context. A commonly used choice principle is the class choice scheme, which asserts that every $V$-indexed definable family of hyperclasses has a choice function $\mathcal F:V\to \Cl$. Historically, this principle can be traced to Marek's PhD thesis in the 1970s \cite{marek:phd} and was rediscovered several times, most recently in the work of Hrb{\'a}\v{c}ek on nonstandard second-order set theories with infinitesimals \cite{hrbacek:nonstandardClassSetTheory} and Antos and Friedman on definable hyperclass forcing over $\KM$ models \cite{antos:thesis}. To give the precise statement of the class choice scheme, we use the standard notation that if $Z$ is a class, then
$$Z_x=\set{y\mid(x,y)\in Z}$$
denotes the class coded on section $x$ of $Z$. We think of such a class $Z$ as naturally encoding the function $x\mapsto Z_x$, enabling us in effect to refer to functions $\mathcal F:V\to\mathcal S$ as above. In general, if a collection of classes can be coded like this by a single class $Z$, we will say that it constitutes a \emph{codable} hyperclass.

\begin{definition}
The \emph{class choice} scheme is the following scheme of assertions, for any formula $\varphi$ in the language of second-order set theory and allowing any class parameter $A$:
$$\forall x\,\exists X\,\varphi(x,X,A)\rightarrow\exists Z\,\forall x\,\varphi(x,Z_x,A).$$
\end{definition}
As we hinted in the introduction, the class choice scheme can also be viewed as a collection principle rather than a choice principle, because it is equivalent over \GBC\ to the following slightly weaker-seeming principle, the \emph{class-collection} scheme:
$$\forall x\,\exists X\,\varphi(x,X,A)\rightarrow \exists Z\,\forall x\,\exists y\,\varphi(x,Z_y,A).$$
In this formulation, the class witnesses $X$ are merely collected as sections $Z_y$, but not necessarily on the corresponding section $Z_x$. Note first that either principle over \GBC\ implies \KM, and then the further point is that with the class collection scheme one can choose for each $x$ a suitable $y$ by global choice and thereby replace the sections on the correct index, making the two principles equivalent.

The class choice scheme admits several natural weakenings, by stratifying on the logical complexity of $\varphi$ or restricting to sets of indices.\goodbreak

\begin{definition}\
\begin{itemize}
\item The $\Sigma^1_n$ (or $\Pi^1_n$)-\emph{class choice} scheme is the restriction of the class choice scheme to instances $\varphi$ of complexity at most $\Sigma^1_n$ (or $\Pi^1_n$).\\
\item The \emph{set-indexed class choice} scheme is the restriction of the class choice scheme to \emph{set}-indexed families. That is, for any set $a$, class parameter $A$, and any formula $\varphi$ in the language of second-order set theory, we include
$$\forall x\in a\,\exists X\,\varphi(x,X,A)\rightarrow\exists Z\,\forall x\in a\,\varphi(x,Z_x,A).$$
\item The \emph{parameter-free class choice} scheme is the restriction of the class choice scheme to parameter-free formulas $\varphi(x,X)$.
\end{itemize}
\end{definition}

We find it insightful to compare the class choice scheme with its analogue in second-order arithmetic, as we discuss in section~\ref{sec:soArithmetic}. For instance, the notion of the parameter-free class choice scheme was introduced by Guzicki \cite{guzicki:choiceScheme} in the context of second-order arithmetic.

%Two natural applications of the class choice scheme are to establish the \Los\ theorem for internal second-order ultrapowers and to show that formulas of the form $\forall x\,\varphi(x)$ and $\exists x\,\varphi(x)$, where $\varphi$ is $\Sigma^1_n$ (or $\Pi^1_n$) are equivalent to $\Sigma^1_n$ (or $\Pi^1_n$)-formulas. Indeed, it is not difficult to see (theorem~\ref{th:LosIsEquivalentToSetSizedChoiceScheme}) that the set-indexed class choice scheme is precisely equivalent (over \KM) to the \Los\ theorem for internal second-order ultrapowers.

\subsection{The main results}

We aim to show that $\KM$ fails to prove even the weakest instances of the class choice scheme, concerning the existence of choice functions for $\omega$-indexed parameter-free first-order definable families. This provides an interesting counterpoint to the situation in second-order arithmetic, where $\fa$, usually considered a tight arithmetic analogue of $\KM$ (more on this in section~\ref{sec:soArithmetic}), proves all $\Sigma^1_2$ instances of the class choice scheme. We also show that $\KM$ together with the set-indexed class choice scheme does not imply the class choice scheme even for parameter-free first-order formulas and that $\KM$ together with the parameter-free class choice scheme does not imply the class choice scheme. We shall freely make use of mild large cardinal assumptions when proving the main results, since these assumption do not detract from our main point that \KM\ does not prove the desired features for a foundational second-order theory of sets and classes, as set theorists want to use the foundational theory with those and far stronger large cardinals.

\begin{theorem}
\emph{($\ZFC\,+\,$there is a Mahlo cardinal)} There is a model of $\KM$ in which an instance
$$\forall n{\in}\omega\,\exists X\,\varphi(n,X)\rightarrow \exists Z\,\forall n{\in}\omega\,\varphi(n,Z_n)$$
of the class choice scheme fails for a first-order formula $\varphi(x,X)$.
\end{theorem}

\begin{theorem}
\emph{($\ZFC\,+\,$there is a Mahlo cardinal)} There is a model of $\KM$ in which the set-indexed class choice scheme holds, but the class choice scheme fails for a parameter-free first-order formula.
\end{theorem}

\begin{theorem}\label{th:parameter-freeChoice}
\emph{($\ZFC\,+\,$there is an inaccessible cardinal)} There is a model of $\KM$ in which the parameter-free class choice scheme holds but the class choice scheme fails for a $\Pi^1_1$-formula.
\end{theorem}

\noindent Theorem~\ref{th:parameter-freeChoice} and its proof are motivated by the work in~\cite{guzicki:choiceScheme}.

These failures of $\KM$ to prove even the weakest instances of the class choice scheme indicates a fundamental weakness in the theory and might lead us to reconsider its prominent role in second-order set theory. To highlight the weakness, we show that the \Los\ theorem scheme can fail for the internal second-order ultrapower of the $\KM$ universe. 

\begin{theorem}
\emph{($\ZFC\,+\,$there is an inaccessible cardinal)} There is a model of $\KM$ whose internal second-order ultrapower by an ultrafilter on $\omega$ is not itself a $\KM$-model.
\end{theorem}

Similarly, we show that \KM\ does not prove that large cardinal embeddings, such as the ultrapower of the universe by a normal measure on a measurable cardinal, are elementary in the language of \KM. 

\begin{theorem}
\emph{($\ZFC\,+\,$measurable cardinal with an inaccessible above)} There is a model of \KM\ with a measurable cardinal $\delta$, such that the internal ultrapower of the universe by a normal measure on $\delta$ is not elementary in the language of \KM\ set theory.
\end{theorem}

We also show that $\KM$ 
does not prove that $\Sigma^1_1$ formulas are (up to equivalence) closed under first-order quantifiers, even bounded first-order quantifiers. 

\begin{theorem}
\emph{($\ZFC\,+\,$measurable cardinal with an inaccessible above)} The theory $\KM$ fails to establish that $\forall \eta{<}\delta\,\psi(\eta,X)$ is provably equivalent to some $\Sigma^1_1$ formula whenever $\psi$ is.
\end{theorem}

In every case, these various deficiencies of $\KM$ are remedied by replacing it with the theory which we call $\KM^+$, which augments \KM\ with the class choice scheme. A result of Marek shows that $\KM^+$ is bi-interpretable with the first-order theory $\ZFC^-$, that is, $\ZFC$ without the powerset axiom,\!\footnote{More precisely $\ZFC^-$ consists of the axioms of $\ZFC$ without the powerset axiom, with the collection scheme instead of the replacement scheme, and with the well-ordering principle instead of the axiom of choice. See \cite{GitmanHamkinsJohnstone2016:WhatIsTheTheoryZFC-Powerset?} for details.}
together with the assertion that there is a largest cardinal that is furthermore an inaccessible cardinal \cite{marek:phd}.
\section{About Kelley-Morse set theory}\label{sec:KM}
The first axiomatic foundation for second-order set theory, the \Godel-Bernays axioms $\GBC$, arose out of the work of Bernays, \Godel\, and Von Neumman in the 1930s, and essentially codified the basic requirements outlined in the introduction. It consisted of the $\ZFC$ axioms for sets, class replacement, and first-order comprehension. Wang and Morse first considered adding second-order comprehension to $\GBC$, and the resulting theory, which appeared in Kelley's text on general topology \cite{kelley:topology}, became known as the Kelley-Morse axioms. Below we give a precise description of the axiomatizations.

In what follows, we will use the abbreviation $V$ for the class of all sets in a model of a sufficiently strong second-order set theory.
\begin{definition}\
\begin{itemize}
\item Let \textbf{Set} be the following collection of axioms for sets in the language of second-order set theory: extensionality, foundation, emptyset,  infinity, powerset, union.
\item The \textbf{extensionality} axiom asserts that any two classes having the same elements are equal.
\item The \textbf{replacement} axiom asserts that whenever $F$ is a function and $a$ is a set, then $F\image a$ is a set.
\item The \textbf{global choice} axiom asserts that there is a global choice function $C:V\smallsetminus\set{\emptyset}\to V$ such that for every nonempty set $a$, $C(a)\in a$.
\item Let $\Gamma$ be some collection of formulas in the language of second-order set theory. The \textbf{comprehension} axiom scheme for $\Gamma$ asserts that any collection of sets definable by a formula in $\Gamma$ with class parameters is a class.
\end{itemize}
\end{definition}
\noindent In its modern formulation, $\GBC$ consists of \textbf{Set}, \textbf{extensionality}, \textbf{replacement}, \textbf{global choice}, and \textbf{comprehension} for first-order formulas. The theory $\KM$ consists of $\GBC$ augmented by full second-order \textbf{comprehension}. Note that the axioms of \textbf{Set} together with \textbf{global choice} and \textbf{replacement} guarantee that $V$ satisfies $\ZFC$ even in the expanded language with predicates for finitely many classes. Note also that \textbf{global choice} is equivalent over the remaining axioms of $\GBC$ to the existence of a well-ordering of $V$ and indeed to the existence of an $\Ord$-enumeration of $V$.

The constructible universe $L$ together with its definable collections is a model of $\GBC$, demonstrating the equiconsistency of $\GBC$ and $\ZFC$. More so $\GBC$ is conservative over $\ZFC$ for first-order assertions, meaning that any property of sets provable in $\GBC$ is already provable in $\ZFC$. Meanwhile, the more powerful comprehension scheme of \KM\ gives it considerably greater strength. The theory \KM\ proves that there is a transitive model of $\ZFC$, indeed that there are many such transitive models---it proves that the first-order universe $V$ is the union of an elementary chain of $V_\delta$ for a club class of cardinals $\delta$. These cardinals will have $V_\delta\satisfies\ZFC$ and so they are worldly, but being in an elementary chain, they will also be otherworldly in the sense of \cite{Hamkins2020:Otherworldly-cardinals} and indeed totally otherworldly and much more. Meanwhile, if $\kappa$ is an inaccessible cardinal, then $\<V_\kappa,\in,V_{\kappa+1}>$ is a model of \KM. So the consistency strength of \KM\ is strictly below \ZFC\ plus one inaccessible cardinal, but strictly above \ZFC\ plus a proper club class of totally otherworldly cardinals.

For completeness, let us provide the elementary chain argument, which provides the proper class of otherworldly cardinals (see also  \cite{MarekMostowski:onExtendibilityOfZFModels}). Central to this argument is the notion of a first-order \emph{truth predicate} class in a model of \KM.
Given a model $\V=\<V,\in,\Cl>$ of \KM, we define that a class $\Tr\in \Cl$ is a first-order \emph{truth predicate} if it is a collection of \Godel\ codes $\ulcorner \varphi(\overline a)\urcorner$ with set parameters $\overline a$ such that all atomic truths of $V$ are in $\Tr$ and $\V$ thinks that $\Tr$ obeys the recursive Tarskian truth conditions for Boolean combinations and first-order quantification. By Tarski's theorem on the undefinability of truth, such a collection $\Tr$ can never be first-order definable over $V$. But indeed, we show below that every model of $\KM$ has a first-order truth predicate class. Note also that if $V$ has a nonstandard $\omega$, then the truth predicate $\Tr$ must provide a coherent conception of truth even for nonstandard \Godel\ codes, obeying the Tarski recursion for standard and nonstandard formulas. It is easy to see by meta-theoretic induction that every truth predicate will agree with actual truth on all standard formulas.
\begin{lemma}
If $\V=\<V,\in,\Cl>\models\KM$, then $\Cl$ contains a first-order truth predicate class $\Tr$.
\end{lemma}
\begin{proof}
Since $\Sigma_0^0$-truth is definable, $\Cl$ contains a $\Sigma_0^0$-truth predicate and, since a $\Sigma_{n+1}^0$-truth predicate is definable from a $\Sigma_n^0$-truth predicate for any $n\in\omega$ (whether $n$ is standard or not), it follows that if $\Cl$ contains a $\Sigma_n^0$-truth predicate, then it must contain a $\Sigma_{n+1}^0$-truth predicate. Now we apply induction to the formula asserting that there is a $\Sigma_n^0$-truth predicate and obtain that there is a $\Sigma_n^0$-truth predicate for every $n\in\omega$. Note that induction for second-order assertions follows from the comprehension scheme of $\KM$. The $\Sigma_n^0$-truth predicate, if it exists, is unique, since there can be no smallest formula for which the disagreement arises, and therefore we can use the $\Sigma_n^0$-truth predicates to define $\Tr$, which is then a class by comprehension.
\end{proof}

The argument easily adapts to class parameters, and so every model of $\KM$ has truth predicates for $V$ relativized to any finite number of classes. As we mentioned in the introduction, the proof does not work in $\GBC$, first because we cannot apply induction to a second-order assertion (so we don't necessarily get $\Sigma^0_n$ truth predicates for nonstandard $n$) and second because even if $\Sigma_n^0$-truth predicates exist for every $n\in\omega$, we cannot unify them into a single class without second-order comprehension. As we mentioned, if $M$ is a transitive model of $\ZFC$ with a definable well-ordering of the universe, then augmenting $M$ with its definable classes is a model of $\GBC$ that has $\Sigma_n^0$-truth predicates for every $n\in\omega$ (because it has a standard $\omega$ and $\Sigma_n^0$-truth is definable), but by Tarski's theorem on the non-definability of truth it will not have a class that is a first-order truth predicate.

Let us show next that whenever one has a truth predicate for first-order truth in a model of \KM, then the truth predicate will declare even the nonstandard axioms of \ZFC\ to be true. This conclusion is straightforward for the standard-finite instances, since the truth predicate agrees with actual truth on standard sentences, but when the model is $\omega$-nonstandard, it is a bit subtler to handle the nonstandard instances of the \ZFC\ axioms. For any model $V\models\ZFC$, let $\ZFC^V$ denote the collection of \Godel\ codes of the $\ZFC$ axioms from the perspective of $V$. If $V$ has a nonstandard $\omega$, then $\ZFC^V$ differs from $\ZFC$ by having additional nonstandard instances of replacement (we can disregard separation because it follows from replacement).

\begin{lemma}
If $\V=\<V,\in,\Cl>\models\KM$ and $\Tr$ is a first-order truth predicate in $\Cl$, then $\Tr$ will declare every sentence of $\ZFC^V$ to be true.
\end{lemma}

\begin{proof}
The main subtle issue here is that although $\Tr$ is correct about actual truth and this includes the actual \ZFC\ since $V\models\ZFC$, nevertheless $\ZFC^V$ might include further nonstandard instances of axioms when $V$ is $\omega$-nonstandard. It suffices to verify that $\Tr$ contains all instances of replacement, including any nonstandard instances. So suppose the hypothesis of a replacement assertion $\ulcorner \forall x\in a\,\exists !y\,\varphi(x,y,c)\urcorner$ is in $\Tr$. Using that $\Tr$ satisfies Tarskian truth conditions, it follows that
$$\V\models\forall x\in a\exists !y\,\ulcorner \varphi(x,y,c)\urcorner \in \Tr.$$
But this is an instance of actual replacement (one involving the class parameter $\Tr$), and therefore $\V$ satisfies that there exists a set $b$ collecting all the $y$. Thus, $\Tr$ must contain the assertion
$$\ulcorner \exists b\,\forall y\,(y\in b\leftrightarrow \exists x\in a\,\varphi(x,y,c))\urcorner$$
and so $\Tr$ validates replacement for the (possibly nonstandard) formula $\ulcorner \varphi(x,y,c)\urcorner$.
\end{proof}

We can now argue that every model of $\KM$ has what it thinks is a transitive model of $\ZFC$, and what is more, it realizes the universe $V$ as the union of a continuous elementary chain of rank-initial segments $V_\kappa\prec V$. 

\begin{lemma}\label{le:transitiveModelKM}
If $\V=\<V,\in,\Cl>\models\KM$, then $V$ is the union of a continuous elementary chain of a class club of rank-initial segments
$$V_{\delta_0}\prec V_{\delta_1}\prec\cdots\prec V_{\delta_\xi}\prec\cdots\prec V$$
In particular, $V$ thinks that each $V_{\delta_\xi}\models\ZFC$. Thus, there is a closed unbounded class of totally otherworldly cardinals.
\end{lemma}

\begin{proof}
By the usual Levy-Montague reflection, we can find a class club of $\Sigma_1^0$-elementary substructures
$$\<V_\delta,\in,\Tr>\prec_{\Sigma_1} \<V,\in,\Tr>$$
for arbitrarily large $\delta$. Note that if $\<V_\delta,\in,\Tr>\prec_{\Sigma_1} \<V,\in,\Tr>$, then the structure $\<V_\delta,\in,\Tr>$ knows that $\Tr$ satisfies the Tarskian truth conditions.
So truth in $V_\delta$ agrees with truth in $V$, and consequently $V_\delta\prec V$, and also $V$ sees that $V_\delta\models\ZFC^V$, which means that $\delta$ is a worldly cardinal. Indeed, since $V_\delta\prec V_\lambda$ for all $\delta<\lambda$ in the class club, these cardinals are all totally otherworldly (and in fact much more).
\end{proof}

The argument adapts to any class parameters, so we actually get a club of totally $X$-otherworldly cardinals for any class $X$. 

Let us also record the upper bound. 

% But in particular, we see that \KM\ implies that there is a class club of totally otherworldly cardinals, which has consistency strength strictly far greater than $\Con(\ZFC)$ and $\Con(\ZFC+\Con(\ZFC))$ and so forth. Meanwhile, it is also known that \KM\ is weaker in consistency strength than an inaccessible cardinal, which we now show. Recall that $\KM^+$ is the theory $\KM$ together with the class choice scheme.

\begin{lemma}\label{le:inaccessibleKM}
{\rm(\ZFC)} If $\kappa$ is an inaccessible cardinal, then $\<V_\kappa,\in,V_{\kappa+1}>\models\KM^+$. 
\end{lemma}

\begin{proof}
Since $\kappa$ is inaccessible, $V_\kappa\models\ZFC$. The \textbf{global choice} axiom holds in $\<V_\kappa,\in,V_{\kappa+1}>$ because $V$ satisfies choice, and therefore $V_{\kappa+1}$ has a choice function for $V_\kappa$, and the \textbf{replacement} axiom holds because $V_\kappa$ contains all its small subsets. \textbf{Comprehension} holds because separation holds for $V_\kappa$ in $V$ and $V_{\kappa+1}$ has all subsets of $V_\kappa$. The class choice scheme holds using collection together with choice in $V$. \end{proof}

It follows by the \Lowenheim-Skolem theorem that $V_\kappa$ will have countable models of $\KM^+$, and so the existence of an inaccessible cardinal implies $\Con(\KM^++\Con(\KM^+))$ and much more.

If we seek merely a model of $\KM$ but not necessarily $\KM^+$, that is, without the class choice scheme, then we do not need \ZFC\ in the background, but only \ZF, as follows.

\begin{lemma}\label{le:ZFmodelOfKM}
Assume \ZF. If $\kappa$ is an inaccessible cardinal and $V_\kappa$ is well-orderable, then $\<V_\kappa,\in,V_{\kappa+1}>\models\KM$.
\end{lemma}

\noindent Since we may not have choice functions for all families of subsets of $V_{\kappa+1}$, the class choice scheme can fail in $\<V_\kappa,\in,V_{\kappa+1}>$. So if we are hoping to obtain a model of $\KM$ in which the class choice scheme fails, this suggests a natural strategy, which we will implement in section~\ref{sec:separatingChoiceSchemes}, namely, to look for such $\ZF$ models, which we will find as symmetric inner models of forcing extensions, obtained by forcing at or above the inaccessible cardinal $\kappa$.

We should also note that the consistency of $\KM$ follows from the consistency of $\ZFC^-$ together with the assertion that there is an inaccessible cardinal. If $N\models\ZFC^-$ and $\kappa$ is inaccessible in $N$, then $\<V_\kappa^N,\in,P^N(V_\kappa)>\models\KM^+$, where $P^N(V_\kappa)$ is the collection of all subsets of $V_\kappa^N$ that are elements of $N$. This is the easy direction of the result mentioned in the introduction, $\KM^+$ is bi-interpretable with $\ZFC^-$ together with the assertion that there is a largest cardinal, which is inaccessible.

%It will be part of our later strategy when aiming to produce various countermodels of \KM to apply this lemma inside the various \ZF\ models that arise as symmetric extensions. We will start with a model of \ZFC\ having an inaccessible cardinal, force at or above this cardinal to $V[G]$, and then form the symmetric extension $N\of V[G]$, in which $\kappa$ will remain inaccessible, and then apply lemma \ref{le:ZFmodelOfKM}.

\section{A detour into second-order arithmetic}\label{sec:soArithmetic}

Many of the theorems of second-order set theory are direct analogues of their counterparts in the more extensively studied field of second-order arithmetic. Second-order arithmetic is a formal framework for studying numbers together with sets of numbers, which we can think of as the reals. Because of this its main role has been providing the weakest possible foundations for formalizing analysis. It is interpretable in a two-sorted logic with the language that consists of the first-order language of arithmetic $\mathcal L=\set{+,\cdot,<,0,1}$ together with a membership relation $\in$ for numbers in sets. Relevant models of second-order arithmetic have the form $\mathcal M=\<M,+,\cdot,<,0,1,\Cl>$, where $\<M,+,\cdot,<,0,1>$ is a model of first-order arithmetic and $\Cl$ is a collection of subsets of $M$. Axiomatic foundations for second-order arithmetic consist of the first-order axioms $\PA^-$ (Peano Axioms without induction), the induction scheme for first-order formulas with set parameters, set existence principles, and possibly strengthening of the induction scheme to second-order assertions. Primarily, the axiomatizations differ in their set existence principles. For instance, the second-order theory ${\rm WKL_0}$ asserts that any collection of numbers recursive in a set is itself a set and if a set of numbers codes an infinite binary tree, then there is another set of numbers coding some infinite path through this tree. The analogue of $\GBC$ in second-order arithmetic is the theory ${\rm ACA_0}$ which asserts that any first-order definable collection of numbers (with set parameters) is a set and the analogue of $\KM$ is one of the strongest second-order arithmetic theories $\fa$, which has full second-order comprehension. If $V\models\ZF$, then its model of first-order arithmetic $\<\omega,+,\cdot,<,0,1>$ together with all its definable subsets is a model of ${\rm ACA_0}$ and $\<\omega,,+,\cdot,<,0,1,P(\omega)>\models\fa$.

The natural analogue of the class choice scheme for second-order arithmetic asserts for every formula in the language of second-order arithmetic and set parameter $A$ that:
$$\forall n\,\exists X\,\varphi(n,X,A)\rightarrow \exists Z\,\forall n\,\varphi(n,Z_n,A).$$ We can also analogously define the $\Sigma^1_n$ (or $\Pi^1_n$)-choice scheme fragments as well as the parameter-free choice scheme. If $V\models\ZF+\AC_\omega$, then the class choice scheme holds in its model of second-order arithmetic $\<\omega,,+,\cdot,<,0,1,P(\omega)>$. If $N\models \ZFC^-$ + every set is countable, then $\<\omega,+,\cdot,<,0,1,P^N(\omega)>\models\fa$ + choice scheme.

Although technically sophisticated set theoretic techniques are used to determine how much of the class choice scheme follows from $\fa$, it is easy to see that ${\rm ACA}_0$ doesn't prove even first-order instances of the class choice scheme.
\begin{observation}\label{obs:ACA0ChoiceScheme}
There is a model of ${\rm ACA}_0$ in which the class choice scheme fails for a first-order formula.
\end{observation}
\begin{proof}
Let $\Cl$ be the collection of all definable subsets of $\<\omega,+,\cdot,<,0,1>$, which in particular includes its $\Sigma_n$-truth predicates for every $n\in\omega$, but not the full truth predicate. Thus, the model $\<\omega,+,\cdot,<,0,1,\Cl>\models{\rm ACA}_0$, but it cannot collect the $\Sigma_n$-truth predicates.
\end{proof}
A model $\mathcal M$ of $\fa$ (or even a much weaker second-order theory) can code in a surprising amount of set theory. Every well-founded tree coded in $\Cl$ itself codes some hereditarily countable set and therefore we can consider the model of first-order set theory consisting of these sets (see \cite{simpson:second-orderArithmetic}, chapter 7 for details of the coding). Indeed, if $\mathcal M\models \fa$ + choice scheme, then it is not difficult to verify that its associated model of set theory satisfies $\ZFC^-$ together with the assertion that every set is countable, making it that the two theories are bi-interpretable. In particular, the model of set theory associated to $\mathcal M$ has its own constructible hierarchy $L^{\mathcal M}$ and the assertion that some tree codes a set of the form $L_\alpha$ has complexity $\Pi_1^1$ in $\mathcal M$ (the second-order quantifier is necessary to verify well-foundedness). More generally, the model of set theory has submodels of the form $L^{\mathcal M}[A]$ for every set $A$ in $\Cl$. These constructible and relatively constructible sets of $\mathcal M$ have all their usual properties such as a definable well-ordering, and they even satisfy Shoenfield's absoluteness theorem: $\Sigma^1_2$ assertions $\varphi(A)$ are absolute to $L[A]$. Details of these arguments can be found in \cite{simpson:second-orderArithmetic} (chapter 7). Using Shoenfield's absoluteness it immediately follows that $\fa$ proves the $\Sigma^1_2$-choice scheme.
\begin{theorem}\label{th:choiceSchemeHoldsArithmetic}
The theory $\fa$ proves all instances of the $\Sigma^1_2$-choice scheme.
\end{theorem}
\begin{proof}
Suppose that $\mathcal M=\<M,+,\cdot,<,0,1,\Cl>\models\fa$ and $\forall n\,\exists X\,\varphi(n,X,A)$ holds in $\mathcal M$ for some $A\in \Cl$ and $\varphi$ of complexity $\Sigma^1_2$. By Shoenfield's absoluteness, since $\exists X\,\varphi(n,X,A)$ is $\Sigma^1_2$, for every $n$, there is $X\in L^{\mathcal M}[A]$ such that $\varphi(n,X,A)$ holds. But since $\mathcal M$ has a definable well-ordering $<^{L[A]}$ of $L^{\mathcal M}[A]$ , for every $n$, there is a $<^{L[A]}$-least such $X$. So we can assume that the witnessing classes $X$ are unique and therefore we can collect them using comprehension.
\end{proof}
It turns out that theorem~\ref{th:choiceSchemeHoldsArithmetic} is optimal because it is possible for a $\Pi^1_2$-instance of the class choice scheme to fail in a model of $\fa$. The strategy for producing such a model is the same as suggested earlier for the case of $\KM$: we look for a model of $\ZF$ in which countable choice fails for some family that also happens to be definable in its model of second-order arithmetic. The result is due to Feferman and \Levy\ who used their classic Feferman-\Levy\ model of $\ZF$ which demonstrated that in the absence of countable choice $\omega_1$ can be a countable union of countably many sets (see \cite{levy:choicescheme}, Theorem 8). Let's recall their construction. Working in $L$, consider the finite-support product $\P=\Pi_{n\in\omega}\Coll(\omega,\omega_n)$. Let $G\subseteq \P$ be $L$-generic and $G_n$ be the restrictions of $G$ to $\overline{\P}_n=\Pi_{m<n}\Coll(\omega,\omega_n)$. Since each $\Coll(\omega,\omega_n)$ is weakly homogeneous, we can use the construction outlined in section~\ref{subsec:symmetricModels} of the appendix to construct a symmetric submodel $N$ of $L[G]$ with the property that its sets of ordinals are precisely those added by proper initial segments of the collapsing product. This is the Feferman-\Levy\ model. Thus, all $\omega_n$ are countable in $N$, but $\omega_\omega$ is uncountable because no initial segment $\overline{\P}_n$ can collapse it, making it the $\omega_1$ of $N$. Note also $P^N(\omega)=\Union_{n\in\omega}P^{L[G_n]}(\omega)$.
\begin{theorem}[Feferman-\Levy]
There is a model of $\fa$ in which the class choice scheme fails for a $\Pi^1_2$-formula.
\end{theorem}
\begin{proof}
Let $\mathcal M=\<\omega,+,\cdot,<,0,1,P^N(\omega)>$, where $N$ is a Feferman-\Levy\ model. Each $L_{\omega_n}$ is coded in $P^N(\omega)$, but $L_{\omega_\omega}$ is not. The assertion that a set $X$ codes $L_{\omega_n}$ (for some fixed $n$) has complexity $\Pi^1_2$ because it says that $X$ codes an $L_\alpha$ and whenever $Y$ codes an $L_\beta$ with $\beta>\alpha$, then $L_\beta$ agrees that $\alpha=\omega_n$, where as we observed previously the assertion that a set codes an $L_\alpha$ has complexity $\Pi_1^1$. Thus, the $\Pi^1_2$ instance of the class choice scheme
$$\forall n\,\exists X\, X\text{ codes }L_{\omega_n}\rightarrow\exists Z\,\forall n\,Z_n\text{ codes }L_{\omega_n},$$
fails in $\mathcal M$ because if such $Z$ existed, $\mathcal M$ would be able to construct $L_{\omega_\omega}$.
\end{proof}

\section{\KM\ does not prove the class choice scheme}\label{sec:independence}

% we've mentioned this several times already, so I have omitted this here.
% Before we get to the technical arguments showing that $\KM$ does not imply the class choice scheme, let's observe that it is easy to make the class choice scheme fail in a $\GBC$ model, using an analog of the argument used for ${\rm ACA}_0$ in observation~\ref{obs:ACA0ChoiceScheme}.

% \begin{observation}
% If there is an $\omega$-model of $\ZFC$, then there is a model of $\GBC$ in which the class choice scheme fails for a first-order formula.
% \end{observation}

% \begin{proof}
% Let $M$ be an $\omega$-model of $\ZFC$ and consider the $\GBC$ model consisting of $L^M$ together with its definable subsets. Then, because $L^M$ has a standard $\omega$, $\Sigma_n$-truth is definable in $L^M$ for all its natural numbers $n$, but truth is not definable.
% \end{proof}

Let us now prove that \KM\ does not prove the class choice scheme. For a first proof, we implement a straightforward generalization of the Feferman-\Levy\ construction with forcing to collapse the first $\omega$ many successor cardinals of an inaccessible $\kappa$ to $\kappa$. This will show that a $\Pi_1^1$-instance of the class choice scheme can fail in a model of $\KM$, in contrast to the situation with second-order arithmetic, where $\fa$ proves the $\Sigma^1_2$-choice scheme.

To explain, suppose that $L$ has an inaccessible cardinal $\kappa$. Working in $L$, consider the finite-support product $\P=\Pi_{n\in\omega}\Coll(\kappa,\kappa^{(+n)})$ to collapse the first $\omega$ many successor cardinals of $\kappa$ to $\kappa$. Let $G\subseteq\P$ be $L$-generic and let $G_n$ be the restriction of $G$ to the first $n$ factors $\overline{\P}_n=\Pi_{m<n}\Coll(\kappa,\kappa^{(+m)})$.  We now construct a symmetric submodel $N$ of $L[G]$ with the property that its sets of ordinals are precisely those added by proper initial segments of the collapsing product (see further details in the construction outlined in section~\ref{subsec:symmetricModels} of the appendix). Thus, all $\kappa^{(+n)}$ have size $\kappa$ in $N$, but $\kappa^{(+\omega)}$ is the $\kappa^+$ of $N$ because it is not collapsed by any proper initial segment $\overline{\P}_n$. Although $V_\kappa^{L[G]}$ has new subsets of $\omega$ and indeed all $(\kappa^{+n})$ are collapsed to $\omega$, nevertheless these are not added by the initial segments of the forcing, since the finite products $\bar\P_n$ are ${<}\kappa$-closed. In particular, $V_\kappa^N=V_\kappa=L_\kappa$.
% this seems all beside the point:
% To see this, we argue by induction on $\alpha\leq\kappa$ that $V_\alpha^N=V_\alpha$. Since $V_\omega^N=V_\omega$, we can suppose inductively that $V_\alpha^N=V_\alpha$ for some $\alpha$. Fix $A\in V_{\alpha+1}^N$. Since $L$ has a bijection between $V_\alpha$ and some ordinal, then so does $N$, meaning that we can suppose without loss that $A$ is a subset of ordinals. But then, for some $n$, we have $A\in V_{\alpha+1}^{L[G_n]}=V_{\alpha+1}$. 
Indeed, what we have is   $V_{\kappa+1}^N=\Union_{n\in\omega}V_{\kappa+1}^{L[G_n]}$. Clearly $\kappa$ is still inaccessible in $N$ and $N$ has a well-ordering of $V_\kappa$ because $L$ did. Thus, we have satisfied the hypothesis of lemma~\ref{le:ZFmodelOfKM}, and so $\<V_\kappa,\in,V_{\kappa+1}^N>\models\KM$.

Using the same coding as described in the previous section for models of $\fa$, the classes of a model $\mathcal M\models \KM$ can encode a structure of sets, including new ordinals, that sits above the sets of $\mathcal M$.\footnote{For more details on the coding, see \cite{antos:thesis}.} For models of the form $\<V_\kappa,\in, V_{\kappa+1}>$, these ordinals are precisely $\kappa^+$ and the extra sets are those of $H_{\kappa^+}$.
\begin{theorem}\label{th:FailureChoiceSchemePi11}
\emph{($\ZFC\,+\,$there is an inaccessible cardinal)} There is a model of $\KM$ in which the class choice scheme fails for a $\Pi_1^1$-formula.
\end{theorem}
\begin{proof}
Let $\mathcal M=\<V_\kappa,\in,V_{\kappa+1}^N>$, where $N$ is a symmetric model as described above. Each $L_{\kappa^{(+n)^L}}$ is coded in $V_{\kappa+1}^N$, but $L_{\kappa^{(+\omega)^L}}$ is not, since $\kappa^{(+\omega)^L}$ is not collapsed in $N$. The assertion that a class codes $L_{\kappa^{(+n)^L}}$ is $\Pi_1^1$ because to check for well-foundedness is now $\Pi^1_0$ as opposed to $\Pi_1^1$ in the case of models of $\fa$. Thus, the $\Pi^1_1$ instance of the class choice scheme
$$\left[\forall n\in\omega\,\exists X\, X\text{ codes }L_{\kappa^{(+n)^L}}\right]\rightarrow\exists Z\,\forall n\in\omega\ \bigl(Z_n\text{ codes }L_{\kappa^{(+n)^L}}\bigr),$$
fails in $\mathcal M$ because if such $Z$ existed, $\mathcal M$ would be able to construct $L_{\kappa^{(+\omega)^L}}$, which it cannot even in $N$.
\end{proof}

We shall now aim to improve theorem~\ref{th:FailureChoiceSchemePi11} by developing techniques that will enable us to show that the class choice scheme can fail in a model of $\KM$ even in a first-order instance. The general idea will be to produce a model $V\models\ZFC$ with an inaccessible cardinal $\kappa$ and $\omega$ many $\kappa$-Suslin trees with the property that forcing with any finite subcollection of them does not add branches to the rest. We can then consider the symmetric model arising from forcing with the product of the trees in the same fashion as the one built from the product of the collapse forcing and having the property that its sets of ordinals are precisely those added by the initial segments of the forcing. The symmetric model will then have branches for every one of the $\omega$ many trees, but it won't be able to collect them.

Suppose that $\kappa$ is a regular cardinal. A \emph{$\kappa$-tree} is a partial order such that the predecessors of any node are well-ordered of order type less than $\kappa$, with every such order type realized, and there are fewer than $\kappa$ many elements on any particular level. Such a tree is $\kappa$-\emph{Aronzajn}, if it has no branches of size $\kappa$, and it is $\kappa$-\emph{Suslin} tree if it also has no antichains of size $\kappa$. A tree is said to be \emph{homogeneous} if for any two nodes on the same level, there is an automorphism mapping one to the other. Such a tree is weakly homogeneous as a partial order in the sense of appendix subsection~\ref{subsec:symmetricModels}.

Using lemma~\ref{le:trees}, we can assume that we are working in a universe $V$ with an inaccessible cardinal $\kappa$ which has an $\omega$-sequence $\vec T=\<T_n\mid n<\omega>$ of homogeneous $\kappa$-Suslin trees with the property that for every $n<\omega$, the product $\Pi_{m<n} T_m$ has the $\kappa$-cc. It follows that forcing with any such product does not add branches to any of the later trees $T_i$ with $i\geq n$, since this would violate the chain condition for the product including that later tree. Consider the full-support product $\P=\Pi_{n\in\omega}T_n$ and notice that, because each $T_n$ is weakly homogeneous when viewed as a partial order, $\P$ is the type of poset described in subsection~\ref{subsec:symmetricModels} of the appendix. Let $G\subseteq \P$ be $V$-generic and $G_n$ be the restriction of $G$ to $\P_n=\Pi_{m<n}T_m$. So we can obtain a symmetric model $N$ of $V[G]$ with the property that its sets of ordinals are precisely those added by the proper initial segments of the product $\P$. Thus, every tree $T_m$ has a branch in $N$, but there is no collecting set with a branch of every tree $T_m$ because an initial segment $\P_n$ adds branches only to those trees. Since $\P$ is $\ltkappa$-distributive by lemma~\ref{le:trees}, we have $V_\kappa^{V[G]}=V_\kappa=V_\kappa^N$ and so it follows that $V_{\kappa+1}^N=\Union_{n<\omega}V_{\kappa+1}^{V[G_n]}$. Clearly $\kappa$ is inaccessible in $N$ and $N$ has a well-ordering of $V_\kappa$. Thus, we have satisfied the hypothesis of lemma~\ref{le:ZFmodelOfKM}, and so $\<V_\kappa,\in,V_{\kappa+1}^N>\models\KM$.

\begin{theorem}\label{th:firstOrderParameterFailure}
\emph{($\ZFC\,+\,$there is an inaccessible cardinal)} There is a model of $\KM$ in which an instance
$$\forall n{\in}\omega\,\exists X\,\varphi(n,X,A)\rightarrow \exists Z\,\forall n{\in}\omega\,\varphi(n,Z_n,A)$$
of the class choice scheme fails for a first-order formula $\varphi$ with a class parameter~$A$.
\end{theorem}

\begin{proof}
Let $\mathcal M=\<V_\kappa,\in,V_{\kappa+1}^N>$, where $N$ is the symmetric model constructed as above. The sequence $\vec T$ is coded in $V_{\kappa+1}^N$ and each $T_m$ has a branch in $V_{\kappa+1}^N$, but $V_{\kappa+1}^N$ cannot have a class collecting at least one branch from every $T_m$. The assertion that a class is a branch through the tree $T_n$ is $\Pi_0^1$ in the parameter $\vec T$. Thus, the $\Pi^1_0$ instance of the class choice scheme
$$\psi(\vec T):=\forall n\in\omega\,\exists X\,X\text{ is a branch of }T_n\rightarrow\exists Z\,\forall n\in\omega\, Z_n\text{ is a branch of }T_n$$
fails in $\mathcal M$.
\end{proof}
The failing instance of the class choice scheme produced in the proof of theorem~\ref{th:firstOrderParameterFailure} uses the parameter $\vec T$, which is the sequence of trees whose branches cannot be collected. A slightly more involved construction starting with a Mahlo cardinal eliminates the parameter and produces a first-order failure of the parameter-free choice scheme. The idea is to code the sequence $\vec T$ into the continuum function below $\kappa$ before forcing with the trees to obtain the required symmetric submodel. We will argue that following the coding forcing, the sequence $\vec T$ retains its key original properties: the trees $T_n$ are $\kappa$-Aronzajn, forcing with any initial segment of them is $\ltkappa$-distributive and does not add branches to the later trees.\goodbreak

\begin{theorem}\label{firstOrderParameterFreeFailure}
\emph{($\ZFC\,+\,$there is a Mahlo cardinal)} There is a model of $\KM$ in which an instance
$$\forall n{\in}\omega\,\exists X\,\varphi(n,X)\rightarrow \exists Z\,\forall n{\in}\omega\,\varphi(n,Z_n)$$
of the parameter-free choice scheme fails for a first-order formula $\varphi(x,X)$.
\end{theorem}

\begin{proof}
By lemma~\ref{le:trees}, we can work in a universe $V$ that has a Mahlo cardinal $\kappa$ and an $\omega$-sequence $\vec T=\<T_n\mid n<\omega>$ of homogeneous $\kappa$-Suslin trees such that the product of any initial segment of them has the $\kappa$-cc. Because the forcing to add the sequence $\vec T$ is $\ltkappa$-distributive, we can also assume that $\GCH$ holds below $\kappa$ (by forcing to add the trees to a model of $\GCH$). Let $A$ be a subset of the regular cardinals below $\kappa$ coding $\vec T$ in some definable manner. Let $\s$ be the Easton-support product of length $\kappa$ which forces with $\Add(\alpha,\alpha^{++})$ for every regular $\alpha\in A$ and let $H\subseteq \s$ be $V$-generic. By lemma~\ref{le:EastonProductKappaCC}, $\s$ has the $\kappa$-cc and the usual arguments show furthermore that $\kappa$ remains Mahlo in $V[H]$. Each $T_n$ remains a $\kappa$-Aronzajn tree in $V[H]$ because by lemma~\ref{le:MahloEastonProductNewSubsets}, $\s$ cannot add a new subset of $\kappa$, all whose initial segments are in $V$. Now let's consider a product $\P_n=\Pi_{m<n}T_m$ and let $G_n\subseteq \P_n$ be $V$-generic. Since $\s\times \P_n$ is a product, $V[H][G_n]=V[G_n][H]$. Observe that $\kappa$ is Mahlo in $V[G_n]$ because $\P_n$ has the $\kappa$-cc in $V$ and $\s$ continues to be an Easton-support product in $V[G_n]$ because $\P_n$ is $\ltkappa$-distributive in $V$. Thus, $\s$  has the $\kappa$-cc in $V[G_n]$ and it follows that $\kappa$ remains Mahlo in $V[G_n][H]$. In particular, $\kappa$ is regular there, and so it follows that $\P_n$ continues to be $\ltkappa$-distributive in $V[H]$. Finally, no $T_i$ with $i\geq n$ can have a branch in $V[G_n][H]$ because it doesn't have it in $V[G_n]$ and $\s$ cannot add a branch to $T_i$ by lemma~\ref{le:MahloEastonProductNewSubsets}. This completes the argument that the sequence $\vec T$ continues to have all the properties in $V[H]$ necessary to carry out the proof of theorem~\ref{th:firstOrderParameterFailure}, but now we don't need to use $\vec T$ as a parameter to refer to the trees $T_n$.
\end{proof}
\section{Separating fragments of the class choice scheme over \KM}\label{sec:separatingChoiceSchemes}
In this section, we separate the class choice scheme from the set-indexed class choice scheme by showing that the set-indexed class choice scheme can hold in a model of $\KM$, while a first-order instance of the class choice scheme fails. We shall also separate the class choice scheme from the parameter-free choice scheme by showing that the parameter-free choice scheme can hold in a model of $\KM$, while a $\Pi_1^1$-instance of the class choice scheme fails.
\begin{theorem}
\emph{($\ZFC\,+\,$there is an inaccessible cardinal)} There is a model of $\KM$ in which the set-indexed class choice scheme holds, but a first-order instance of the class choice scheme fails.
\end{theorem}
\begin{proof}
Using lemma~\ref{le:trees}, we may assume that we are working in a universe $V$ with an inaccessible cardinal $\kappa$ which has a $\kappa$-sequence $\vec T=\<T_\xi\mid \xi<\kappa>$ of homogeneous $\kappa$-Suslin trees with the property that for every $\xi<\kappa$, the product $\Pi_{\eta<\xi} T_\eta$ has the $\kappa$-cc. %and so in particular forcing with it does not add branches to any $T_\eta$ with $\eta\geq \xi$. 
It follows that none of these products add branches through the later trees $T_\eta$ for $\eta\geq\xi$, since such a branch would prevent the $\kappa$-cc of the product including that later tree. 

Consider now the full bounded-support product $\P=\Pi_{\xi<\kappa}T_\xi$. Let $G\subseteq \P$ be $V$-generic and $G_\xi$ be the restriction of $G$ to $\P_\xi=\Pi_{\eta<\xi}T_\xi$. Since the conditions described in the section~\ref{subsec:symmetricModels} of the appendix still apply, we can obtain a symmetric model $N$ of $V[G]$ with the property that its sets of ordinals are precisely those added by the proper initial segments of the product $\P$. Thus, every tree $T_\eta$ has a branch in $N$, but there is no collecting set with a branch of every tree $T_\eta$. Since $\P$ is $\ltkappa$-distributive by lemma~\ref{le:trees}, we have $V_\kappa^{V[G]}=V_\kappa=V_\kappa^N$ and so it follows that $V_{\kappa+1}^N=\Union_{\xi<\kappa}V_{\kappa+1}^{V[G_\xi]}$. Clearly $\kappa$ is inaccessible in $N$ and $N$ has a well-ordering of $V_\kappa$, and so by lemma~\ref{le:ZFmodelOfKM} we conclude $\<V_\kappa,\in,V_{\kappa+1}^N>\models\KM$. It is clear that the class choice scheme fails in this model for the first-order assertion (with parameter $\vec T$) that each tree $T_\xi$ has a branch. 

It remains to show that the set-indexed class choice scheme holds. For this, suppose that $\<V_\kappa,\in,V_{\kappa+1}^N>\models\forall x\in a\,\exists X\,\varphi(x,X,A)$ for some $a\in V_\kappa$ and $A\in V_{\kappa+1}^N$. Working in $V[G]$, choose for each $i\in a$, a witnessing $X_i\in \Union_{\xi<\kappa}V_{\kappa+1}^{V[G_\xi]}$. Since $a$ has size less than $\kappa$ and $\kappa$ is regular in $V[G]$, there is some $\xi<\kappa$ such that all $X_i$ are in $V[G_\xi]$. But now since the tail forcing $\Pi_{\xi<\eta<\kappa}T_\xi$ is $\ltkappa$-distributive, it follows that the sequence $\<X_i\mid i\in a>$ is itself in $V[G_\xi]$ and hence it is coded in $V_{\kappa+1}^{V[G_\xi]}$. Thus, this code is in $V_{\kappa+1}^N$, and so we can realize the set-indexed instance of the class choice scheme.
\end{proof}
Again, by starting with a stronger hypothesis, we can eliminate the parameter from the class choice scheme failure.
\begin{theorem}
\emph{($\ZFC\,+\,$there is a Mahlo cardinal)} There is a model of $\KM$ in which the set-indexed class choice scheme holds, but a first-order instance of the parameter-free choice scheme fails.
\end{theorem}
\begin{proof}
This is just a direct analogue of the argument given in the proof of theorem~\ref{firstOrderParameterFreeFailure}. Namely, we code the sequence of trees into the \GCH\ pattern below $\kappa$ in such a way that it becomes definable in $V_\kappa$ without parameters.
\end{proof}
The next theorem is motivated by an analogous result established by Guzucki in~\cite{guzicki:choiceScheme} for models of second-order arithmetic.
\begin{theorem}
\emph{($\ZFC\,+\,$there is an inaccessible cardinal)} There is a model of $\KM$ in which the parameter-free choice scheme holds, but a $\Pi_1^1$-instance of the class choice scheme fails.
\end{theorem}
\begin{proof}
Suppose that $L$ has an inaccessible cardinal $\kappa$. Working in $L$, consider the bounded-support product $\P=\Pi_{\xi<\kappa^+}\Coll(\kappa,\kappa^{(+\xi)})$ to collapse the first $\kappa^+$-many successors of $\kappa$ to $\kappa$. Let $G\subseteq\P$ be $V$-generic and $G_\xi$ be the restriction of $G$ to $\overline \P_\xi=\Pi_{\eta<\xi}\Coll(\kappa,\kappa^{(+\eta)})$. Let $N$ be a symmetric model of $L[G]$ with the property that its sets of ordinals are precisely those added by proper initial segments of the product $\P$. Thus, all $\kappa^{(+\xi)}$ in $L$ for $\xi<\kappa^+$ have size $\kappa$ in $N$, but the $\kappa^{(+\kappa^+)}$ of $L$ is the $\kappa^+$ of $N$ because it is not collapsed by any of the proper initial segments $\overline{\P}_\xi$. As in the proof of theorem~\ref{th:FailureChoiceSchemePi11} we may see that $V_\kappa^N=V_\kappa$ and $V_{\kappa+1}^N=\Union_{\xi<\kappa^+}V_{\kappa+1}^{V[G_\xi]}$. Clearly $\kappa$ is regular in $N$ and $N$ has a well-ordering of $V_\kappa$, and so again by lemma~\ref{le:ZFmodelOfKM} we know that $\mathcal M=\<V_\kappa,\in,V_{\kappa+1}^N>\models\KM$.

Let us demonstrate that the class choice scheme fails in $\mathcal M$. Let $<^*$ be an ordering of $\kappa$ in order type $(\kappa^+)^L$ in $N$ and let $\alpha_\xi$ for $\xi<\kappa$ denote the ordinal coded by $\xi$ in $<^*$, meaning $<^*$ restricted to the $<^*$-predecessors of $\xi$. Thus, the $\Pi_1^1$-instance of the class choice scheme
$$\psi(<^*):=\forall \xi\,\exists X\, X\text{ codes }L_{\kappa^{(+\alpha_\xi)}}\rightarrow \exists Z\,\forall \xi\,X\text{ codes }L_{\kappa^{(+\alpha_\xi)}}$$
fails in $\mathcal M$.

Next, we argue that parameter-free choice scheme holds in $\mathcal M$. Suppose that $\mathcal M\models\forall x\,\exists X\,\phi(x,X)$. We can assume, by working below some condition if necessary, that $\one_{\P}$ forces this statement of the canonical name $\dot {\mathcal M}$ for $\mathcal M$. That $\mathcal M$ has such a canonical name follows from the fact that $V_{\kappa+1}^N=\Union_{\xi<\kappa^+}V_{\kappa+1}^{L[G_\xi]}$. First, we argue that for every $x$, there is $\xi_x<\kappa^+$ such that it is forced by $\one_{\P}$ that there is $A\in L[\dot G_{\xi_x}]$ such that $\dot{\mathcal M}\models \varphi(x,A)$. To see this, suppose $x\in V_\kappa$. Since by assumption there is $A\in \Union_{\xi<\kappa^+}V_{\kappa+1}^{L[G_\xi]}$ for which $\mathcal M\models\varphi(x,A)$, we can fix some $\xi_x$ such that there is $A\in L[G_{\xi_x}]$ so that $\mathcal M\models\varphi(x,A)$. And we may fix a condition $p$ forcing that there is such an $A\in L[\dot G_{\xi_x}]$ with $\dot{\mathcal M}\models\varphi(x,A)$. Now suppose towards a contradiction that there is some other condition $q$ which forces that there is no such $A$ in this particular extension $L[\dot G_{\xi_x}]$. Let $\pi$ be a coordinate-respecting automorphism, as described in section~\ref{subsec:symmetricModels} of the appendix, such that $\pi(p)$ and $q$ are compatible. The condition $\pi(p)$ therefore forces that there is $A$ in $L[\pi(\dot G_{\xi_x})]$ for which $\pi(\dot{\mathcal M})\models\varphi(x,A)$. But, because $\pi$ is coordinate respecting, it follows that $\pi(\dot G_\xi)_H$ and $(\dot G_\xi)_H$ are interdefinable from $\pi$ in any extension $L[H]$, and so $L[(\dot G_{\xi_x})_H]=L[\pi(\dot G_{\xi_x})_H]$. Now, it follows that $\pi(\dot{\mathcal M})=\dot{\mathcal M}$ in any such $L[H]$. Thus, $\pi(p)$ forces that there is $A\in L[\dot G_{\xi_x}]$ such that $\dot{\mathcal M}\models\varphi(x,A)$, but this contradicts our assumption that $q$, which is compatible to $\pi(p)$, forces the negation of this statement. This completes the argument that for every $x\in V_\kappa$, there is $\xi_x<\kappa^+$ such that $\one_{\P}$ forces that there is $A\in L[\dot G_{\xi_x}]$ such that $\dot{\mathcal M}\models\varphi(x,A)$. Since $\kappa^+$ is regular in $L$, we can choose a bound $\xi<\kappa^+$ for the $\xi_x$. Thus, $L[G_\xi]$ has witnesses $A$ for every $x\in V_\kappa$. But $V_{\kappa+1}^{L[G_\xi]}$ is has size $\kappa$ in $N$ and therefore is coded by some element of $V_{\kappa+1}^N$. This code gives us the desired collecting class of witnesses for $\varphi(x,X)$.
\end{proof}

In short, the model $\mathcal M$ satisfies the full choice scheme without parameters, but there is a $\Pi^1_1$-instance of the scheme with parameters that goes awry.

\section{Further weakness in \KM}\label{sec:weaknessesKM}

In this section, we argue that $\KM$ as a theory fails to establish some other expected and desirable results in second-order set theory, such as the \Los\ theorem for internal second-order ultrapowers, the elementarity of familiar large cardinal embeddings, the preservation of \KM\ by ultrapowers generally, and the fact that second-order logical complexity is not affected by first-order quantifiers. These various failures, meanwhile, are addressed when we augment $\KM$ with the class choice scheme, thereby forming the strictly stronger theory we denote by $\KM^+$, a theory which does prove those desirable results and which in this sense and others provides a more robust foundation than $KM$ for second-order set theory.

Consider first the case of ultrapowers and the \Los\ theorem, which play a central role role in set theory, particularly in the theory of large cardinals, where one typically undertakes ultrapowers of the entire set-theoretic universe. We would seem naturally to want our foundational theories to handle this central construction robustly. One can indeed mount the ultrapower construction of the universe in second-order set theory. Namely, if $\mathcal V=\<V,\in,\Cl>\models\KM$ and $U\in V$ is an ultrafilter on some set $D$, then we use $U$ to construct the internal second-order ultrapower $\Ult(\mathcal V,U)$ as follows. The sets of $\Ult(\mathcal V,U)$ are provided, of course, by the equivalence classes $[f]_U$ of functions $f\in V$ with domain $D$ with respect to the equivalence relation $f=_Ug$, holding whenever $\set{d\in D\mid f(d)=g(d)}\in U$. This is a congruence with respect to the ultrapower membership relation $[f]_U\in_U [g]_U$, which holds whenever $\set{d\in D\mid f(d)\in g(d)}\in U$. The second-order objects, meanwhile, the classes of $\Ult(\mathcal V,U)$, are taken as the equivalence hyperclasses $[F]_U$ of codable hyperclass functions $F:D\to \Cl$, defined as predicates so that $[f]_U\in [F]_U$ whenever $\set{d\in D\mid f(d)\in F(d)}\in U$. As mentioned earlier, we represent such functions $F$ by a class $\tilde F\of D\times V$ for which $F(d)$ is the slice $\tilde F_d$. It should be noted here that there is no analogue of Scott's trick to choose a unique representative of a hyperclass equivalence class.  Because the filter we used is an element of $V$, the entire construction we have just described can be undertaken inside $\mathcal V$ and we refer to it as the \emph{internal} ultrapower construction. 

One wants not merely to undertake the ultrapower construction, of course, but also to establish the expected desirable features of it. So let us say that the ultrapower $\Ult(\mathcal V,U)$ fulfills the \emph{\Los\ theorem scheme} for internal second-order ultrapowers when every instance of the following biconditional holds, for every assertion $\varphi$ in the language of second-order set theory:
 $$\Ult(\mathcal V,U)\satisfies\varphi([f]_U,[F]_U)\quad\text{if and only if}\quad \set{d\in D\mid \mathcal V\satisfies\varphi(f(d),F(d))}\in U.$$ 
The \Los\ scheme implies the elementarity of the ultrapower embedding, which maps sets $a\mapsto [c_a]_U$ and classes $X\mapsto[c_X]_U$ to the equivalence classes of the corresponding constant functions.
 
One can easily establish every first-order instance of the \Los\ scheme by meta-theoretic induction in $\KM$, that is, for the first-order formulas $\varphi$ only.  We shall presently prove, however, that the second-order instances of the \Los\ scheme are \emph{not} generally provable in $\KM$. Meanwhile, the full second-order scheme is provable in $\KM^+$, and indeed in light of the role played by the axiom of choice in proving the classical first-order \Los\ theorem, it will not be terribly surprising that in fact the \Los\ theorem scheme for internal second-order ultrapowers is precisely equivalent over $\KM$ to the set-indexed class choice scheme.

\begin{theorem}\label{th:LosIsEquivalentToSetSizedChoiceScheme}
The \Los\ theorem scheme for internal second-order ultrapowers is equivalent over $\KM$ to the set-indexed class choice scheme.
\end{theorem}

\begin{proof}
The usual proof of the classical \Los\ theorem adapts easily to prove the \Los\ theorem scheme for internal second-order ultrapowers in any model of $\KM$ that satisfies the set-indexed class choice scheme. One proceeds by meta-theoretic induction on formulas, and the class choice scheme is used to assemble witnesses for the second-order existential quantifier case, just as one uses \AC\ to prove the classical first-order \Los\ theorem. 

For the converse direction, assume $\mathcal V=\<V,\in,\Cl>$ satisfies $\KM$ as well as the \Los\ theorem scheme for the internal second-order ultrapowers. We shall prove that the set-indexed class choice scheme holds as well, proving this by induction on the size of the index set. Assume by induction that all instances of this hold for index sets of size less than $\kappa$, and suppose that we have an instance on $\kappa$ itself. So we assume that for every $\xi<\kappa$ there is $X$ such that $\psi(\xi,X,A)$, where $A$ is some fixed class parameter. We seek to assemble such various witnesses into a coded class sequence. 

Let $U$ be an ultrafilter on $\kappa$ concentrating on tail sets and consider the internal ultrapower $\Ult(\mathcal V,U)$. By our minimality assumption on $\kappa$, it follows that for every $\alpha<\kappa$ we have 
 $$\mathcal V\models\exists Z\ \forall\xi{<}\alpha\
  \psi(\alpha,Z_\xi,A).$$
That is, for any particular $\alpha<\kappa$ we can unify $\alpha$ many witnesses below $\alpha$ into a coded class $Z$ whose sections $Z_\xi$ work for all $\xi<\alpha$. Since this assertion is true for every $\alpha$ less than $\kappa$ and the set of such $\alpha$ is in $U$, it follows by the \Los\ theorem scheme that 
 $$\Ult(\mathcal V,U)\satisfies \exists Z\ \forall\xi{<}[\text{id}]_U\ \psi(\xi,Z_\xi,[c_A]_U),$$ 
where $\text{id}:\kappa\to\kappa$ is the identity function and $c_A:\kappa\to\Cl$ is the constant function with value $A$. Let $Z=[F]_U$ be the class witnessing this second-order existence assertion. So $F:\kappa\to\Cl$ is coded in $\Cl$ and for $U$-almost-every $\alpha$ we have that 
 $$\mathcal V\satisfies \forall\xi{<}\alpha\ \psi(\xi,F(\alpha)_\xi,A).$$
Since $U$ concentrates on tail segments, there are unboundedly many $\alpha$ like that. That is, unboundedly often, $F(\alpha)$ provides an $\alpha$-sequence of classes $F(\alpha)_\xi$ that work for every $\xi<\alpha$. We can therefore assemble these witnesses into a single $\kappa$ sequence of suitable witnesses. Namely, let $F'(\xi)$ be $F(\alpha)(\xi)$ for the smallest $\alpha$ for which this works as a witness for $\psi(\xi,F'(\xi),A)$. The function $F'$ therefore witnesses this instance of the class choice scheme, as desired.
\end{proof}

Thus we have established that the \Los\ theorem scheme will fail in any model of $\KM$ that is not a model of the set-indexed class choice scheme. We claim next, however, that the situation is much worse than this. Namely, we claim that an internal second-order ultrapower of a $\KM$-model without the set-indexed class choice scheme is not necessarily even a model of $\KM$ itself. In short, without the set-indexed choice scheme, internal ultrapowers do not preserve $\KM$.

To begin with this argument, let us record here our earlier observation that the first-order fragment of the \Los\ theorem scheme does hold for internal ultrapowers, even allowing class parameters (but not class quantifiers). This can be proved by the usual induction on formulas. We had needed the class choice scheme only to handle the class quantifier case of the induction.

\begin{lemma}\label{le:first-orderLos}
The \Los\ theorem scheme holds in internal ultrapowers of models $\mathcal V=\<V,\in,\Cl>\satisfies\GBC$ for all first-order assertions $\varphi$ with class parameters. 
 $$\Ult(\mathcal V,U)\satisfies\varphi([f]_U,[F]_U)\quad\text{if and only if}\quad \set{d\in D\mid \mathcal V\satisfies\varphi(f(d),F(d))}\in U.$$ 
In particular, the internal ultrapower map is $\Sigma_1^1$-elementary.
\end{lemma}

Let us begin to expose the problems for \KM\ by proving that  internal ultrapowers of models of \KM\ do not necessarily preserve $\KM$.

\begin{theorem}\label{thm:ultrapowerOfKMModelIsNotKMMOdel}
\emph{($\ZFC\,+\,$there is an inaccessible cardinal)} There is a model of $\KM$ whose internal second-order ultrapower by an ultrafilter on $\omega$ is not a model of $\KM$.
\end{theorem}
\begin{proof}
Let $\mathcal M=\<V_\kappa,\in,V_{\kappa+1}^N>$ be the model of $\KM$ constructed in the proof of theorem~\ref{th:firstOrderParameterFailure}, which has a class sequence $\vec T$ of $\omega$ many $\kappa$-trees, each tree a class in $\mathcal M$ and each tree having a branch in $\mathcal{M}$, but there is no class collecting a branch for every tree. Let $\varphi(n,X,\vec T)$ be a first-order formula expressing that $X$ codes an $n$-sequence of branches for the first $n$-many of the trees in $\vec T$. Fix a nonprincipal ultrafilter $U$ on $\omega$ in $V_\kappa$ and consider $\Ult(\mathcal M,U)$. If $n<\omega$, then $\mathcal M$ satisfies the $\Sigma_1^1$-formula $\varphi(n,X,\vec T)$, and hence by lemma~\ref{le:first-orderLos}, $\Ult(\mathcal M,U)$ satisfies $\exists X\,\varphi([c_n]_U,X,[C_{\vec T}]_U)$. Thus, $\Ult(\mathcal M,U)$ satisfies $\exists X\,\varphi([c_n]_U,X,[C_{\vec T}]_U)$ for every actual natural number $n$. If $\Ult(\mathcal M,U)$ is a model of $\KM$, then it satisfies induction for all second-order assertions about the natural numbers, and therefore it must have some nonstandard natural number $[f]_U$ such that $\exists X\,\varphi([f]_U,X,[C_{\vec T}]_U)$ holds, for otherwise it would be able to define the standard cut $\omega$. Assuming this to be the case, it follows that $$\Ult(\mathcal M,U)\models \varphi([f]_U,[F]_U,[C_{\vec T}]_U)$$
for some class $[F]_U$ and so, by lemma~\ref{le:first-orderLos} since $\varphi$ is first-order, the set
$$\set{n<\omega\mid \varphi(n,F(n),\vec T)}$$
is an element of $U$ and therefore cofinal in $\omega$. But this is impossible, because using $F$, we can as in theorem \ref{th:LosIsEquivalentToSetSizedChoiceScheme} collect a branch for every tree in $\vec T$.
\end{proof}

Next, let us show that large cardinal embeddings arising as the internal ultrapower of a model of \KM\ are not necessarily elementary in the language of \KM\ set theory.

\begin{theorem}
\emph{($\ZFC\,+\,$measurable cardinal with an inaccessible above)} There is a model of \KM\ with a measurable cardinal $\delta$, such that the internal ultrapower of the universe by a normal measure on $\delta$ is not elementary in the second-order language.
\end{theorem}\label{th:Measurable-ultrapowers-not-elementary}

\begin{proof}
We start in a model with a measurable cardinal $\delta$ and an inaccessible cardinal $\kappa$ above it. Using the argument of section~\ref{sec:trees} of the appendix, we force over this model to obtain a universe $V$ in which $\delta$ remains measurable, $\kappa$ remains inaccessible and there are now $\delta$-many homogeneous $\kappa$-Suslin trees $\vec T=\<T_\xi\mid\xi<\delta>$ with the properties expressed in lemma~\ref{le:trees}. The forcing to add the trees preserves the measurability of $\delta$ because it adds no new subsets to $\delta$ and indeed it is $\ltkappa$-distributive. Consider the full-support product $\P=\Pi_{\xi<\delta}T_\xi$. Let $G\subseteq \P$ be $V$-generic and let $G_\eta$ be the restriction of $G$ to the product $\overline {\P}_\eta=\Pi_{\xi<\eta}T_\xi$, that is, forcing with the trees up to $\eta$. So, as usual, we can construct a symmetric submodel $N$ of $V[G]$ with the property that its sets of ordinals are precisely those added by some such fragment $\overline{\P}_\eta$ for some $\eta<\kappa$. This gives us a $\KM$-model $\mathcal M=\<V_\kappa,\in,V_{\kappa+1}^N>$, where 
$$V_{\kappa+1}^N=\Union_{\eta<\delta}V_{\kappa+1}^{V[G_\eta]}.$$ 
In $\mathcal M$, we can collect branches for every proper initial segment of $\vec T$, but not for all of them. Fix a normal measure $U$ on $\delta$ in $V_\kappa$ and consider the internal ultrapower $\mathcal W=\Ult(\mathcal M,U)$ by this measure. Since the measure is countably complete, this model is well-founded, and by composing with the Mostowski collapse, we may assume it is a transitive model. 

Let us use the ultrapower of $V[G]$ by $U$ to describe more precisely the nature of $\mathcal W$. Namely, let $j:V\to M$ be the ultrapower by $U$ in $V$. This embedding has critical point $\delta$, but $\kappa$ is an inaccessible cardinal above, so we have $\delta<j(\delta)<\kappa=j(\kappa)$. Since $\P$ is $\leqdelta$-distributive, $j$ lifts uniquely to $j:V[G]\to M[H]$, where $j(G)=H=\set{q\in j(\P)\mid \exists p\in G\ q\leq j(p)}$ and the lift is again the ultrapower by $U$ in $V[G]$. Since the forcing adds no new ${<}\kappa$-sequences, it follows that  $V_\kappa=V_\kappa^{V[G]}$, and consequently $j(V_\kappa)=V_\kappa^M$, and since $V_\kappa$ is closed under functions on $\delta$, the ultrapower of $V_\kappa$ by $U$ is isomorphic to $j(V_\kappa)=V_\kappa^M$. Thus, the collection of sets of the ultrapower $\mathcal W$ is precisely $V_\kappa^M$. Next, we argue that the collection of classes of the ultrapower $\mathcal W$ is the set $\Union_{\eta<\delta}V_{\kappa+1}^{M[H_\eta]}$. Each class in the ultrapower is represented by a function $F:\delta\to V_{\kappa+1}[G_\eta]$ for some $\eta<\delta$, which means that $[F]_U$, when viewed as an element of $M[H]$, is in $V^M_{\kappa+1}[H_\eta]$. Conversely, an element of $M[H]$ that is in $\Union_{\eta<\delta}V_{\kappa+1}^{M[H_\eta]}$ is represented by a function $F:\delta\to V_{\kappa+1}^{V[G_\eta]}$ for some $\eta<\delta$, and $F$ has to be in $V[G_\eta]$ because $\P$ is ${<}\kappa$-distributive. Thus, there is an obvious isomorphism between the classes of $\mathcal W$ and $\Union_{\eta<\delta}V_{\kappa+1}^{M[H_\eta]}$ given by mapping the equivalence class of a function $F$ in $\mathcal W$ to the corresponding equivalence class in $M$, with the ultrapower embedding being given by $j$. We can consider the symmetric model $\bar N$ of $M[H]$, obtained analogously to the symmetric model $N$ above, whose sets of ordinals are precisely those in $M[H_\eta]$ for some $\eta<\delta$, and observe that $\mathcal W=\langle V_\kappa^M,\in, V_{\kappa+1}^{\bar N}\rangle$. Thus, in particular, $\mathcal W$ is a model of $\KM$. 

In $\mathcal M$, for each $\eta<\delta$ we have a class that codes a list of branches for the first $\eta$ many trees $T_\xi$ for $\xi<\eta$. This is a statement about $\vec T$ and $\delta$ in $\mathcal M$. But if we apply the ultrapower embedding, the corresponding statement is not true about $j(\vec T)$ and $j(\delta)$, since in $\mathcal W$ there can be no such list for the first $\delta$ many trees in $j(\vec T)$, since any such list would be represented by a class function $F$, and we could then use $F$ to collect a branch for every $T_\xi$ with $\xi<\delta$. This violates the elementarity of the ultrapower embedding in the second-order language.
\begin{comment}
Since the forcing adds no new ${<}\kappa$-sequences, it follows that  $V_\kappa=V_\kappa^{V[G]}$, and consequently $j(V_\kappa)=V_\kappa^M=V_{\kappa}^{M[j(G)]}$, and these are precisely the sets of the ultrapower $\mathcal W$. The classes of the ultrapower $\Ult(\mathcal M,U)$ are represented by functions $F:\delta\to V_{\kappa+1}^N$ that are coded in $V_{\kappa+1}^N$ and these are therefore in  $V[G_\eta]$ for some $\eta<\kappa$. It follows that the classes of $\mathcal W$ are precisely the subsets of $V_\kappa$ available in some $M[H_\eta]$, for some $\eta<\delta$. 

In $\mathcal M$, for each $\eta<\delta$ we have a class that codes a list of branches for the first $\eta$ many trees $T_\xi$ for $\xi<\eta$. This is a statement about $\vec T$ and $\delta$ in $\mathcal M$. But if we apply the ultrapower embedding, the corresponding statement is not true about $j(\vec T)$ and $j(\delta)$, since in $\mathcal W$ there can be no such list for the first $\delta$ many trees in $j(\vec T)$, since any such list would have to be represented by a function $[F]_U$ that had branches through all the trees $T_\xi$ for $\xi<\delta$, whereas there is no such function available in any $M[H_\eta]$. This violates the elementarity of the ultrapower embedding. 
\end{comment}
\end{proof}
In theorem~\ref{thm:ultrapowerOfKMModelIsNotKMMOdel}, we able to produce a model of $\KM$ having an ultrafilter on $\omega$ whose ultrapower was not itself a $\KM$-model. Although, we constructed a model of $\KM$ with a normal measure whose ultrapower map is not elementary, the ultrapower itself did satisfy $\KM$. This leads to the following natural question to which we suspect that the answer is negative.

\begin{question}
Is the internal ultrapower of a model of $\KM$ by a normal measure on a cardinal $\delta$ in the model always a model of $\KM$?
\end{question}

Notice that the violation of elementarity occurred with an assertion of the form $\forall\eta{<}\delta\,\exists X\, \psi(\eta,X,\vec T)$, where $\psi$ asserts that $X$ codes an $\eta$-sequence of branches through the first $\eta$ many trees $T_\xi$ of $\vec T$. This statement has complexity $\forall\eta{<}\delta\,\Sigma^1_1$. So this brings us immediately to another failure of $\KM$, namely, that it does not prove that the logical complexity of $\Sigma^1_1$ is closed under bounded first-order quantification. 

Let us recall that the first-order set theory $\ZFC$ proves that the logical complexity of a first-order assertion is not affected by subsequent bounded quantifiers applied in front. Such a fact underlies the normal form theorem for assertions in set theory, by which one has all the unbounded quantifiers up front and applying a bounded quantifier does not change the complexity class. It seems desirable that our foundation of second-order set theory would prove analogously such a normal form, where all the second-order quantifiers in a formula appear in the front, and this complexity is not affected by additional subsequent first-order quantifiers. 

This follows easily from $\KM^+$, but is not the case for $\KM$. Let us say that a formula in second-order set theory has complexity $\Sigma^1_1$ if it has the form $\exists X\ \varphi(x,X,A)$, where $\varphi$ has only first order quantifiers. By simple coding we can equivalently allow multiple existential class quantifiers $\exists X\exists Y\ \varphi(x,X,Y,A)$.

\begin{theorem}
\emph{($\ZFC\,+\,$measurable cardinal with an inaccessible above)} The theory $\KM$ fails to establish that $\forall \eta{<}\delta\ \psi(x)$ is provably equivalent to some $\Sigma^1_1$ formula whenever $\psi$ is.
\end{theorem}

\begin{proof}
% I modified this proof in light of the previous theorem, since it is now a simple consequence. JDH
This follows immediately from the proof of theorem \ref{th:Measurable-ultrapowers-not-elementary}, which provided a \KM\ model with a measurable cardinal $\delta$, for which the ultrapower by a normal measure was not elementary for an assertion of the form $\forall\eta{<}\delta\ \psi$, where $\psi$ was $\Sigma^1_1$. But by lemma \ref{le:first-orderLos}, the ultrapower map is $\Sigma^1_1$-elementary, so that formula cannot be provably equivalent to any $\Sigma^1_1$-assertion. 
\end{proof}

\section{Other class choice principles}

Other interesting choice principles for classes arise by analogy with the principles of dependent choice in first-order set theory. Let's first consider the principle which allows us to make $\omega$ many dependent choices for a definable relation on classes without terminal nodes.

\begin{definition}\label{def:omegaDependentChoice}
The \emph{$\omega$-dependent class choice} scheme is the following scheme of assertions, where $\varphi(X,Y,W)$ is any formula in the language of second-order set theory and $A$ is a class parameter
$$\forall X\,\exists Y\,\varphi(X,Y,A)\rightarrow \exists Z\,\forall n{<}\omega\,\varphi(Z_n,Z_{n+1},A).$$
\end{definition}
Over $\KM$, the class choice scheme together with the $\omega$-dependent class choice scheme is equivalent to the following second-order reflection principle.

\begin{definition}
A model of a sufficiently strong second-order set theory $\<V,\in,\Cl>$ satisfies \emph{second-order reflection} if for every second-order assertion $\varphi(X)$, there is a (nonempty) codable hyperclass $\mathcal A$ such that for all $X\in \mathcal A$, we have $\<V,\in,\Cl>\models\varphi(X)$ if and only if $\<V,\in,\mathcal A>\models\varphi(X)$.
\end{definition}

That is, the assertion $\varphi(X)$ is absolute from the full collection of classes $\Cl$ to the coded fragment $\mathcal A$. (Note: there are other forms of second-order reflection in the literature that are not the same as this formulation.)

\begin{theorem}
Over $\KM$, the class choice scheme together with the $\omega$-dependent class choice scheme is equivalent to second-order reflection.
\end{theorem}
\begin{proof}
First, suppose that the class choice scheme together with the $\omega$-dependent class choice scheme hold in a model $\mathcal V=\<V,\in,\Cl>\models\KM$. Fix a second-order formula $\varphi(X)$. Interpreting $\mathcal V$ in first-order logic and appealing to the Tarski-Vaught test, we note that it suffices to find a codable hyperclass $\mathcal A$ with the property that whenever $\exists Y\theta(\vec X,Y)$ is a subformula of $\varphi(X)$ and $\vec A$ are some finitely many classes in $\mathcal A$ such that $\exists Y\,\theta(\vec A,Y)$ holds in $\mathcal V$, then some such $Y$ is in $\mathcal A$. Let $\psi(X,Y)$ be the relation expressing that the collection of classes coded by $Y$ is closed under existential witnesses to subformulas of $\varphi(X)$ with parameters coming from the collection of classes coded by $X$ in the manner described above. It readily follows from the class choice scheme that for every class $X$, there is a class $Y$ such that $\psi(X,Y)$ holds in $\mathcal V$. Thus, by the $\omega$-dependent class choice scheme, there is an $\omega$-sequence $Z=\<Z_n\mid n<\omega>$ such that each $Z_{n+1}$, when viewed as the collection of classes it codes, is closed under existential witnesses for subformulas of $\varphi(X)$, with parameters coming from the classes coded by $Z_n$. But then the codable hyperclass consisting of the union of all the $Z_n$ has the desired property.

Conversely, suppose now that $\mathcal V=\<V,\in,\Cl>$ is a model of second-order reflection. First, we argue that $\mathcal V$ satisfies the class choice scheme. Suppose that $\mathcal V\models\forall x\,\exists X\,\varphi(x,X,B)$. Let $\mathcal A$ be a codable hyperclass such that $\<V,\in,\mathcal A>$ satisfies $\forall x\,\exists X\,\varphi(x,X,B)$ and let $A$ be such that $\mathcal A=\set{A_\xi\mid\xi\in\ORD}$ is the collection of the slices of $A$, which we may assume are indexed by ordinals. Now to every set $x$, we can associate that unique $A_\xi$, where $\xi$ is least such that $\<V,\in,\mathcal A>\models\varphi(x,A_\xi,B)$, and so we can collect the witnessing classes. Next, we argue that the $\omega$-dependent class choice scheme holds in $\mathcal V$. Suppose that $\mathcal V\models\forall X\,\exists Y\,\varphi(X,Y,B)$ and let $\mathcal A$ be a codable hyperclass such that $\<V,\in,\mathcal A>$ satisfies this sentence as well. Again, using some $A$ whose slices code the elements of $\mathcal A$, we can pick out an $\omega$-chain of choices for $\varphi(X,Y)$.
\end{proof}
It was shown in \cite{FriedmanGitman:ModelOfACNotDCInaccessible} that the $\omega$-dependent class choice scheme does not follow from $\KM^+$, and thus, in particular, second-order reflection can fail in models of $\KM^+$.

\begin{comment}
It is not known whether the $\omega$-dependent choice scheme follows from the class choice scheme, meaning that it is also not know whether $\KM^+$ proves second-order reflection. Because models of $\KM$ are bi-interpretable with models of $\ZFC^-$ in which there is a largest cardinal that is furthermore inaccessible, this question is intimately connected to the question posed by Zarach~\cite{Zarach1996:ReplacmentDoesNotImplyCollection} about whether reflection holds in models of $\ZFC^-$, i.e. whether every formula is reflected by some transitive set in the model. If it can be shown that second-order reflection fails in a $\KM^+$-model, then this translates directly to the failure of first-order reflection in the corresponding $\ZFC^-$-model.
\end{comment}
Analogous to how the the principle of dependent choices can be generalized in first-order set theory, we can generalize definition~\ref{def:omegaDependentChoice} to consider the $\kappa$-dependent class choice scheme for any cardinal $\kappa$, as well as the statement that the $\kappa$-dependent class choice scheme holds for all cardinals $\kappa$.
\begin{definition}
If $\kappa$ is a cardinal, then the \emph{$\kappa$-dependent class choice} scheme is the following scheme of assertions, where $\varphi(X,Y,W)$ is any formula in language of second-order set theory and $A$ is a class parameter
$$\forall \beta{<}\kappa\,\forall X:\beta\to \Cl\,\,\exists Y\,\varphi(X,Y,A)\rightarrow \exists Z:\ORD\to \Cl\,\,\forall\beta{<}\kappa\ \varphi(Z\restrict\beta,Z_\beta,A).$$
\end{definition}
Finally, we can consider the $\ORD$-dependent class choice scheme.
\begin{definition}
The \emph{$\ORD$-dependent choice} scheme is the following scheme of assertions, where $\varphi(X,Y,W)$ is any formula in language of second-order set theory and $A$ is a class parameter
$$\forall \beta{\in}\ORD\,\forall X:\beta\to \Cl\,\,\exists Y\,\varphi(X,Y,A)\rightarrow \exists Z:\ORD\to \Cl\,\,\forall\beta{\in}\ORD\,\varphi(Z\restrict\beta,Z_\beta,A).$$
\end{definition}
It is not difficult to see that the $\ORD$-dependent class choice scheme implies the class choice scheme. Supposing that $\forall \xi{\in}\ORD\,\exists X\,\varphi(\xi,X,A)$ holds in a $\KM$-model, we consider the relation $\psi(X,Y,A)$, which asserts that $\varphi(\beta,Y,A)$ holds, where $\beta$ is the domain of $X$. Clearly for every $X:\beta\to \Cl$, there is $Y$ such that $\psi(X,Y,A)$ holds. Now, applying the $\ORD$-dependent class choice scheme, we obtain a class $Z$ coding a sequence of classes such that for every ordinal $\beta$, $\varphi(Z\restrict\beta,Z_\beta,A)$ holds. But then it follows by induction that for all $\beta\in\ORD$, $\varphi(\beta,Z_\beta,A)$ holds.

It is not known at present whether the more general dependent class choice schemes can be separated from the $\omega$-dependent class choice scheme or from each other. 
\begin{comment}
Nothing seems to be currently known about whether any of the dependent choice principles can be separated or whether they follow from the class choice scheme. Friedman, Gitman, and Antos recently observed that $\KM$, $\KM^+$, and $\KM$ together with the $\ORD$-dependent choice scheme are all equiconsistent because given a model $\<V,\in,\Cl>$ of $\KM$, its $L$ together with the ``constructible" classes, meaning those classes that are elements of an $L_X$ constructed according to a well-order $X\in\Cl$, is a model of $\KM$ together with the $\ORD$-dependent choice scheme.
\end{comment}
The $\ORD$-dependent choice scheme was used in the work of Friedman and Antos~\cite{antos:thesis}, who showed that the theory $\KM$ together with the $\ORD$-dependent choice scheme is preserved by all tame definable hyperclass forcing, a result that so far cannot be generalized to just $\KM^+$. Tameness is the necessary and sufficient condition for class or hyperclass forcing to preserve $\ZFC$ for sets. In contrast, Antos showed that both $\KM$ and $\KM^+$ are preserved by all tame class forcing.

\section{Appendix}
\subsection{Symmetric models}\label{subsec:symmetricModels}
The arguments below are standard in symmetric model constructions (see for example~\cite{dimitriou:thesis}, Lemma 1.29) and we include them here only for completeness.

A partial order $\P$ is said to be weakly homogeneous if for any two conditions $p,q\in \P$, there is an automorphism $\pi$ of $\P$ such that $\pi(q)$ is compatible to $p$. Let $\kappa$ be a regular cardinal and $\P_\alpha$ for $\alpha<\kappa$ be weakly homogeneous partial orders. Consider the bounded-support (or full-support) product $\P=\Pi_{\alpha<\kappa}\P_\alpha$. Let us say that an automorphism $\pi$ of $\P$ \emph{respects coordinates} if it is made up of the automorphisms on the individual coordinates, that is if there are automorphisms $\pi_\alpha$ of $\P_\alpha$ such that $\pi(p)=\<\pi_\alpha(p(\alpha))\mid \alpha<\kappa>$. Let $\mathcal G$ be the group of all coordinate-respecting automorphisms of $\P$. Let $\Fil$ be the filter on $\mathcal G$ generated by the subgroups $H_\alpha$ for $\alpha<\kappa$ consisting of all automorphisms $\pi$ such that $\pi_\beta$ is identity for all $\beta<\alpha$, so $\pi$ has a nontrivial action only at $\alpha$ and beyond. If $\sigma$ is a $\P$-name, we call $\sym(\sigma)$ the subgroup of $\mathcal G$ consisting of the automorphisms fixing $\sigma$. A $\P$-name $\sigma$ is \emph{symmetric} if $\sym(\sigma)$ is in $\Fil$, meaning that it contains some $H_\alpha$. Let $\HS^{\Fil}$ be the collection of all hereditarily symmetric $\P$-names.

Given a condition $p\in\P$ and an ordinal $\beta<\kappa$, we will denote by $p^{\restrict\beta}$ the condition $p\restrict\beta$ concatenated with the trivial tail. 
\begin{lemma}
Given the homogeneity assumptions above, suppose that $p\in \P$ and $p\forces\check\xi\in\sigma$ for some hereditarily symmetric $\P$-name $\sigma$ such that $H_\beta\subseteq \sym(\sigma)$, then $p^{\restrict\beta}\forces \check\xi\in\sigma$.
\end{lemma}\label{le:InitialSegmentDecidesMembership}

\begin{proof}
If the conclusion fails, then there is $q\leq p^{\restrict\beta}$ such that $q\forces \check\xi\not\in\sigma$. Since the $\P_\alpha$ are weakly homogeneous, there is an automorphism $\pi\in H_\beta$ such that $\pi(q)$ is compatible with $p$. Since $H_\beta$ is contained in $\sym(\sigma)$, we know $\pi(q)\forces \check\xi\notin \sigma$. But, any condition $r\leq p,\pi(q)$ must force both that $\check\xi\in\sigma$ and that $\check\xi\notin\sigma$, which is a contradiction.
\end{proof}

Let $G\subseteq \P$ be $V$-generic and $G_\alpha$ be the restrictions of $G$ to $\overline{\P}_\alpha=\Pi_{\beta<\alpha}\P_\alpha$. Consider the symmetric submodel $N\of V[G]$ arising from the hereditarily symmetric names:
 $$N=\set{\sigma_G\mid \sigma\in \HS^{\Fil}}.$$
The standard symmetric model arguments show that this is a model of \ZF.

\begin{lemma}\label{le:SubsetsOfOrdinalsOfSymmetricModels}
The sets of ordinals of $N$ are precisely those added by proper initial segments $\overline{\P}_\alpha$ of $\P$, namely those in some $V[G_\alpha]$.
\end{lemma}

\begin{proof}
Suppose that $A$ is a subset of ordinals in $N$. Let $\sigma\in \HS^{\Fil}$ be a $\P$-name such that $\sigma_G=A$. Then some condition $p\in G$ forces that $\sigma$ is a subset of ordinals. Since $\sigma\in \HS^{\Fil}$, there is $\beta$ such that $H_\beta\subseteq\sym(\sigma)$. Consider the $\P$-name  $\sigma^*=\set{\<\check\xi,q^{\restrict\beta}>\mid q\leq p,q\forces\check\xi\in\sigma}$, which can be viewed in the obvious way as a $\overline{\P}_\beta$-name. It suffices to argue that $p\forces\sigma=\sigma^*$. Let $H\subseteq \P$ be some $V$-generic containing $p$ and suppose that $\xi\in \sigma_H$. Then there is $q\leq p$ such that $q\forces \check\xi\in \sigma$ and $q\in H$, from which it follows that  $\<\check\xi,q^{\restrict\beta}>\in \sigma^*$ and $q^{\restrict\beta}\in H$. So we have $\xi\in \sigma^*_H$. Next, suppose that $\xi\in\sigma^*_H$. Then there is a condition $q\forces \check\xi\in \sigma$ such that  $\<\check\xi,q^{\restrict\beta}>\in \sigma^*$ and $q^{\restrict\beta}\in H$. But by lemma~\ref{le:InitialSegmentDecidesMembership}, it follows that $q^{\restrict\beta}\forces\check\xi\in\sigma$, and so $\xi\in\sigma_H$.
\end{proof}

\subsection{Homogeneous $\kappa$-Suslin trees}\label{sec:trees}

Suppose that $\kappa$ is an inaccessible cardinal. Let's describe a forcing $\Q$ to add a homogeneous $\kappa$-Suslin tree. Given an ordinal $\alpha$, we call a subtree of the tree ${}^{\lt\alpha} 2$ a \emph{normal} $\alpha$-\emph{tree} if it has height $\alpha$, splits at every successor level node, and each of its nodes has successors on every higher level. Given a subtree $T$ of ${}^{\lt\alpha} 2$, we say that a collection $\mathscr C$ of automorphisms of $T$ acts \emph{transitively} on levels if for any two nodes on the same level of $T$, there is an automorphism in $\mathscr C$ that maps one to the other, and we say that $T$ is \emph{homogeneous} if it has such a collection of automorphisms. Note, in particular, that a homogeneous tree $T$ is almost homogeneous as a forcing notion. The conditions in $\Q$ are pairs $(T,F)$, where $T$ is a homogeneous normal $\alpha+1$-tree for a limit ordinal $\alpha$ and $F=\<f_\xi\mid\xi<\lambda>$ is an enumeration of automorphisms of $T$ that act transitively on levels, for some $\lambda<\kappa$. The ordering on $\Q$ is that $(T',F')\leq (T,F)$, where $F=\<f_\xi\mid\xi<\lambda>$ and $F'=\langle f'_\xi\mid\xi<\lambda'\rangle$, if $T'$ end-extends $T$ and $\lambda'\geq \lambda$, and furthermore $f'_\xi$ extends $f_\xi$ for all $\xi<\lambda$. Note that an automorphism of a normal $\alpha+1$-tree is determined by how it permutes the nodes on level $\alpha$, and every such permutation generates an automorphism of the tree. The basic idea of the forcing is that conditions specify a commitment to an initial segment of the generic tree that will be created together with a commitment to how the $\xi$th automorphism will act on that part of the tree. 

\begin{lemma}\label{le:tall-extensions}
Every condition $(T,F)$ in $\Q$ has extensions to conditions $(T',F')$ where $T'$ is as tall as desired below $\kappa$.
\end{lemma}

\begin{proof}
Suppose we are given the condition $(T,F)$, where $T$ is an $\alpha+1$ tree. For any $\beta<\kappa$, we can place a copy of the full tree $2^{\leq\beta}$ on top of every node on the top level of $T$, thereby forming an extension tree $T'$. And we can trivially extend each automorphism $f_\xi$ in $F$ to an automorphism of $T'$ by mapping these various copies of $2^{\leq\beta}$ to each other and furthermore add more automorphisms that permute within each $2^{\leq\beta}$ so as to have a level-transitive action. Thus, we have found $(T',F')\leq(T,F)$ where $T'$ has height $\alpha+\beta+1$. 
\end{proof}

\begin{lemma}\label{le:tree-forcing-strat-closed}
The forcing $\Q$ is $\ltkappa$-strategically closed.
\end{lemma}

\begin{proof}
Consider the two-player game of length $\kappa$, where the players build a descending $\kappa$-sequence of conditions $(T_\alpha,F_\alpha)$ in the forcing $\Q$. Our opponent plays first, but we play first at all limit stages. We win the play, if indeed we are able to play all the way to $\kappa$. 

Given the opening move $(T_0,F_0)$, what we are going to do is pick a particular node $t_0$ on the top level of $T_0$, and then systematically extend this node as the play continues. We will promise that the resulting branch survives through all limit stages of the play. To make our first move, we extend $(T_0,F_0)$ to a stronger condition $(T_1,F_1)$, and pick $t_1$ above $t_0$ and on the top level of $T_1$. Similarly, at successor stages of play, if we have just played $(T_\alpha,F_\alpha)$, having specified a node $t_\alpha$ on the top level, and our opponent plays $(T_{\alpha+1},F_{\alpha+1})$, then we extend $t_\alpha$ to a node $t_{\alpha+1}$ on the top level of $T_{\alpha+1}$, and then we can extend further to any stronger condition $(T_{\alpha+2},F_{\alpha+2})$, picking $t_{\alpha+2}$ above $t_{\alpha+1}$ on the top level of $T_{\alpha+2}$. At limit stages $\lambda<\kappa$, we can form the union of the earlier trees $T^*_\lambda=\Union_{\alpha<\lambda}T_\alpha$, which has some limit height $\eta$. The nodes $\<t_\alpha\mid\alpha<\lambda>$ provide a branch through this tree, and we now place a node $t_\lambda$ at level $\eta$ exactly on top of that branch. For the automorphisms, for each $\xi$ for which corresponding automorphisms $f_\xi$ are eventually defined in the lists $F_\alpha$, we take the union $f^*_\xi=\Union_{\alpha<\xi}f^\alpha_\xi$, where $f^\alpha_\xi$ is the $\xi$th automorphism listed in $F_\alpha$, if defined, and observe that it is an automorphism of $T^*_\lambda$. We close the collection of the $f^*_\xi$ under the group operations to obtain a subgroup $\mathscr G$ of the automorphism group of $T^*_\lambda$. We then extend the tree $T^*_\lambda$ to the tree $T_\lambda$ by adding a node on top of the branch $\langle t_\alpha\mid\alpha<\lambda\rangle$ and on top of every image of that branch under the automorphisms in $\mathscr G$. Each automorphism $f^*_\xi$ extends in an obvious fashion to the corresponding automorphism $f_\xi$ of $T_\lambda$. Note that $T_\lambda$ is normal. Note also that because $\mathscr G$ is a group and the top level of $T_\lambda$ consists of the images of a single node under the automorphisms in $\mathscr G$, the group $\mathscr G$ acts transitively on $T_\lambda$. Finally, we let $F_\lambda$ enumerate the automorphisms in $\mathscr G$ by having $F_\lambda(\xi)=f_\xi$ and adding the rest of the automorphisms in $\mathscr G$ afterwards. We then play $(T_\lambda,F_\lambda)$ at this stage of the game. 

In summary, we keep track of a specified branch through the trees as the game play proceeds and the trees grow taller. At limit stages, we make sure to put a point on top of that branch, and we keep as other nodes on that level only the nodes that would arise from the (group generated by the) automorphisms of the tree to which we are already committed. 
\begin{comment}
This will be a normal $\eta+1$ tree, because the images of the nodes $t_\alpha$ fill out the rest of that level, and so the branch can be copied so as to cover any desired node in $T^*$. And our action is level transitive on $T_\lambda$, because we only kept the nodes on the top level that could arise as an image of $t_\lambda$. 
\end{comment}
\end{proof}

It follows that $\Q$ is $\ltkappa$-distributive, and so in particular it will preserve the inaccessibility of $\kappa$. Suppose that $G\subseteq \Q$ is $V$-generic. In $V[G]$, let $\mathcal T$ be the generic tree which is the union of all first coordinates of conditions in $G$. By lemma \ref{le:tall-extensions}, it follows that $\mathcal T$ is a $\kappa$-tree. Also, $\mathcal T$ is homogeneous, as witnessed by the automorphisms provided in the second coordinates. Next, we show that $\mathcal T$ is $\kappa$-Suslin by proving a variant of the standard sealing lemma. Let $\dot {\mathcal T}$ be the canonical name for $\mathcal T$.
\begin{lemma}\label{le:sealing}
If $(T,F)\in\Q$ forces that $\dot A$ is a maximal antichain of $\dot{\mathcal T}$, then there is $(T',F')\leq (T,F)$ which forces that $\dot A$ is contained in $T'$.
\end{lemma}
\begin{proof}
We will do the usual bootstrap argument to build a tree $T'$ such that $(T',F')$ forces that $\dot A$ is contained in $T'$, while ensuring that the top level of $T'$ can be constructed as in the proof of lemma~\ref{le:tree-forcing-strat-closed} using the automorphic images of a single branch $B$. Thus, $B$ must have the property that the image branch of it under any automorphism $F'(\eta)$ has on it some node that is forced by $(T',F')$ to be in $\dot A$.

Fix a (nontransitive) elementary substructure $M\prec H_{\kappa^+}$ of size less than $\kappa$ containing $\Q,\dot A, (T,F)$ with the property that $\kappa\cap M=\alpha$ is an ordinal. Let's also fix a definable bookkeeping function $\varphi:\kappa\to\kappa$ such that each $\beta<\kappa$ appears cofinally often in the range of $\varphi$, which must then be in $M$.

Now we work inside $M$ until further notice. Let $T=T_0$, $F=F_0$. Let $t_0$ be any node on the top level of $T_0$. Let $(T_1,F_1)$ be any condition strengthening $(T_0,F_0)$ with the property that for every $s\in T_0$, there is $t_s\in T_1$ compatible with $s$ such that $(T_1,F_1)\forces t_s\in \dot A$. To argue that such $(T_1,F_1)$ exists, we go to a forcing extension $V[G]$ by $\Q$ with $(T_0,F_0)\in G$ and observe that, since $\kappa$ is regular there and $T_0$ has size less than $\kappa$, there is some level $\alpha$ of generic tree $\mathcal T$ such that for every $s\in T_0$, some element of $A=\dot A_G$ compatible with it appears below that level. Next, we consult the bookkeeping function $\varphi(0)=\eta$. If $\eta\geq\lambda_0$, let $t_1$ be any top level node of $T_1$ above $t_0$. Otherwise, let $s=F_0(\eta)(t_0)$ and let $s'$ be on the top level of $T_1$ above $s$ and $t_s$. Finally, let $t_1=F_1(\eta)^{-1}(s')$. Thus, we have extended the branch $B$ in such a way that the $\eta$th automorphism will map it to stay above an element of $\dot A$.  This procedure should be repeated at all successor stages below $\alpha$. At limit levels $\delta<\alpha$, we construct $(T_\delta,F_\delta)$ as in the proof of lemma~\ref{le:tree-forcing-strat-closed}.

Now back outside $M$, we end up with a sequence $\<(T_\xi,F_\xi)\mid \xi<\alpha>$ of conditions in $\Q$ together with the branch $B=\<t_\xi\mid\xi<\alpha>$ through $T^*=\Union_{\xi<\alpha}T_\xi$. Construct $(T_\alpha,F_\alpha)$ as in the proof of lemma~\ref{le:tree-forcing-strat-closed}. Now we argue that $(T_\alpha,F_\alpha)$ has the desired property, namely that the image of $B$ under any $F_\alpha(\eta)$ has some node $t$ that is forced by $(T_\alpha,F_\alpha)$ to be in $\dot A$. So fix some $\eta$ in the domain of $F_\alpha$ and let $\xi$ be some stage such that $\eta$ is already in the domain of $F_\xi$. Since $\varphi$ has the property that $\eta$ appears cofinally often in the range, there is some stage $\gamma>\xi$ such that $\varphi(\gamma)=\eta$. At this stage, we chose some $s'$ on the top level of $T_\gamma$ above a node that is forced to be in $\dot A$ and let $t_\gamma=F_\gamma(\eta)^{-1}(s')$, thus ensuring that the image of $B$ under $F_\alpha(\eta)$ extending $F_\gamma(\eta)$ will sit above a node that is forced to be in $\dot A$.
\end{proof}
\begin{lemma}\label{le:ltkappaClosedDenseSubset}
The two-step iteration $\Q*\dot{\mathcal T}$ has a dense subset that is $\ltkappa$-closed.
\end{lemma}
\begin{proof}
Conditions in $\Q*\dot{\mathcal T}$ are pairs $((T,F),\dot a)$, where $\dot a$ is a $\Q$-name for an element of $\dot {\mathcal T}$. Now consider the subset $\s$ of $\Q*\dot{\mathcal T}$ consisting of conditions $((T,F),a)$, where $T$ is an $\alpha+1$-tree and $a$ is a node on level $\alpha$ of $T$. Let $((T,F),\dot a)$ be any condition in $\Q*\dot{\mathcal T}$. Then there is a stronger condition $((T',F'),a')$ with $a'\in T'$ forcing that $\dot a=a'$. Let $a$ be any node above $a'$ on the top level of $T'$. Clearly $((T',F'),a)$ is an element of $\s$ extending $((T,F),\dot a)$. Thus, $\s$ is dense in $\Q*\dot{\mathcal T}$. But since the node $a$ determines a branch through $T$ in a condition $((T,F),a)$ it is clear that $\s$ is $\ltkappa$-closed.
\end{proof}

Now let's consider the bounded-support product $\P=\Pi_{\xi<\kappa}\Q_\xi$ with all $\Q_\xi=\Q$. A straightforward extension of the argument in the proof of lemma~\ref{le:tree-forcing-strat-closed}, where we choose a single branch for every tree appearing in a condition, shows that $\P$ is $\ltkappa$-strategically closed and hence forcing with it preserves the inaccessibility of $\kappa$. Let $\<\mathcal T_\xi\mid\xi<\kappa>$ be the sequence of trees added by $\P$ to a forcing extension $V[G]$. The argument above shows that all the $\mathcal T_\xi$ are $\kappa$-trees. Indeed, all the $\mathcal T_\xi$ are $\kappa$-Suslin and more so, we will argue that the product of any less than $\kappa$ many of them has the $\kappa$-cc. To see this requires an easy generalization of the argument in the proof of lemma~\ref{le:sealing}. Fix $\lambda<\kappa$ and suppose that some condition $\<(T^{(\xi)},F^{(\xi)})\mid\xi<\alpha>$ forces that $\dot A$ is a maximal antichain of the product $\Pi_{\xi<\lambda}\dot{\mathcal T}_\xi$, where $\dot{\mathcal T}_\xi$ are the canonical names for the trees added by $\P$. We work in an elementary substructure $M$ of $H_{\kappa^+}$ of size less than $\kappa$ that is closed under $\lambda$-sequences and has the property that $M\cap\kappa=\alpha$ is an ordinal. In $M$, we fix a bookkeeping function $\varphi:\kappa\to{}^\lambda \kappa$ such that each element of ${}^\lambda \kappa$ appears cofinally often in the range. Because we chose an $M$ that is closed under $\lambda$-sequences, every $f:\lambda\to\alpha$ has to appear cofinally in the range of $\varphi\restrict\alpha$. We build up the desired condition $\<(T'^{(\xi)},F'^{(\xi)})\mid\xi<\alpha'>$ (with $T'^{(\xi)}$ $\alpha+1$-trees) sealing the antichain $\dot A$ by specifying branches $B_\xi$ of $T'^{(\xi)}$ (without the top level) for $\xi<\lambda$ with the property that any sequence of automorphisms $\<\pi_\xi\mid\xi<\lambda>$ with $\pi_\xi$ in the range of $F'^{(\xi)}$ moves branches $B_\xi$ to branches $B^*_\xi$ so that some sequence $\<t^{(\xi)}\mid \xi<\lambda>$ with $t^{(\xi)}\in B^*_\xi$ is forced by $\<(T'^{(\xi)},F'^{(\xi)})\mid\xi<\alpha'>$ to be in $\dot A$. For the additional coordinates beyond $\lambda$ that appear on the conditions as we extend, we simply specify any branch on their tree and use it to build the limit level of the final tree.

It now follows that forcing with a product $\<\mathcal T_\xi\mid\xi<\lambda>$ for some $\lambda<\kappa$ cannot add branches to any tree $\mathcal T_\alpha$ with $\alpha\geq\lambda$. This is because the product $(\Pi_{\xi<\alpha}\mathcal T_\xi)\times \mathcal T_\alpha$ has the $\kappa$-cc and therefore $\Pi_{\xi<\alpha}\mathcal T_\xi$ must force that $T_\alpha$ has the $\kappa$-cc.
The argument given in the proof lemma~\ref{le:ltkappaClosedDenseSubset} also generalizes to show that the two-step iteration $\P*\Pi_{\xi<\kappa}\dot{\mathcal T}_\xi$ has a $\ltkappa$-closed dense subset.

The results can now be collected in the following Lemma.
\begin{lemma}\label{le:trees}
If $\kappa$ is inaccessible or Mahlo in $V$, then $V$ has a forcing extension $V[G]$ in which $\kappa$ remains inaccessible or Mahlo respectively and there are $\kappa$-many homogeneous $\kappa$-Suslin trees $T_\xi$ for $\xi<\kappa$ with the property that for every $\lambda<\kappa$, the full-support product $\Pi_{\xi<\lambda} T_\xi$ has the $\kappa$-cc and is $\ltkappa$-distributive and does not add branches to any $T_\alpha$ with $\alpha\geq\lambda$. Moreover, $\kappa$ remains inaccessible or Mahlo respectively in any forcing extension by $\Pi_{\xi<\kappa}T_\xi$.
\end{lemma}
Note that if $\kappa$ is Mahlo in $V$, then it remains Mahlo in $V[G]$ because ground model stationary subsets of $\kappa$ remain stationary after $\ltkappa$-strategically closed forcing and similarly it remains Mahlo after forcing with $\Pi_{\xi<\kappa}T_\xi$ because the two-step iteration $\P*\Pi_{\xi<\kappa}\dot{\mathcal T}_\xi$ has a dense $\ltkappa$-closed subset.
\subsection{Mahlo cardinals and Easton products}
\begin{lemma}\label{le:EastonProductKappaCC}
If $\kappa$ is Mahlo and $\P=\Pi_{\xi<\alpha}\P_\alpha$ is an Easton product of length $\kappa$, where each $\P_\alpha\in V_\kappa$, then $\P$ has the $\kappa$-cc.
\end{lemma}
\begin{proof}
By chopping off the conditions in $\P$ at the supremum of their support, $\P$ becomes a subset of $V_\kappa$. Consider the structure $\<V_\kappa,\in,\P,A>$ and note that since $\kappa$ is Mahlo, it has an elementary substructure $\<V_\delta,\in,\P\cap V_\delta,A\cap V_\delta>$, where $\delta$ is inaccessible. By our assumption that $A$ has size $\kappa$, $V_\delta$ must miss some element $q$ of $A$. Let $q^*$ be the restriction of $q$ to the supremum of its support in $\delta$, which must be bounded in $\delta$ since $\P$ has Easton support. Since $V_\delta$ satisfies that $A\cap V_\delta$ is a maximal antichain by elementarity, $q^*$ is compatible with some $p\in A\cap V_\delta$. But then $q$ is compatible with $p$ as well, which is the desired contradiction.
\end{proof}
\begin{lemma}\label{le:MahloEastonProductNewSubsets}
Suppose that $\kappa$ is Mahlo and $\P$ is an Easton support product of length $\kappa$, where each $\P_\alpha\in V_\kappa$. If $G\subseteq\P$ is $V$-generic and $A$ is a new subset of $\kappa$ in $V[G]$, then there is an inaccessible cardinal $\delta<\kappa$ such that $A\cap \delta$ is a new subset of $\delta$ in $V[G]$.
\end{lemma}
\begin{proof}
By lemma~\ref{le:EastonProductKappaCC}, $\P$ has the $\kappa$-cc. By chopping off the conditions in $\P$ at the supremum of their support, $\P$ becomes a subset of $V_\kappa$. Fix a nice $\P$-name $\dot A$ such that $\dot A_G=A$, and recall that we can view such a nice name as a function from $\kappa$ into the antichains of $\P$ that sends $\alpha$ to the antichain $\mathcal A_\alpha$, so that $\alpha\in A$ if and only if $G\cap \mathcal A_\alpha\neq\emptyset$. Since $\P$ has the $\kappa$-cc, every $\mathcal A_\alpha\in V_\kappa$. Fix a condition $p\in \P$ forcing that $\dot A\not\in V$. It is easy to check that this is equivalent to assuming that for every $q\leq p$, there is $\alpha<\kappa$ and $r_1,r_2\leq q$ such that $r_1\forces G\cap\mathcal A_\alpha=\emptyset$ and $r_2\forces G\cap \mathcal A_\alpha\neq\emptyset$. The statement that $r\forces G\cap \mathcal A_\alpha\neq\emptyset$ is equivalent to the set $D=\set{q\mid q\leq a\text{ for some }a\in\mathcal A_\alpha}$ being dense below $r$, and the statement $r\forces G\cap \mathcal A_\alpha=\emptyset$ is equivalent to $\forall q\leq r\forall a\in \mathcal A_\alpha\,\neg q\leq a$. Thus, the structure $(V_\kappa,\P,\dot A,p)$ can recognize that $p$ forces that $\dot A$ is new. Since $\kappa$ is Mahlo, there is an inaccessible $\delta$ such that $(V_\delta,\P\cap V_\delta,\dot A\cap V_\delta,p)\prec (V_\kappa,\P,\dot A,p)$. Observe that $\P\cap V_\delta=\P_\delta$ and $\dot A\cap V_\delta=\dot A\restrict\delta$. Thus, by elementarity, $p$ forces in $\P_\delta$ that $\dot A\restrict \delta$ is new. But $p\in G_\delta$, the restriction of $G$ to $\P_\delta$, and therefore $\dot A\restrict\delta$ really is new.
\end{proof}
\bibliographystyle{alpha}
\bibliography{km+}
\end{document}